\documentclass{amsart} 

\usepackage{graphicx}
\usepackage{amsfonts,url}
\usepackage{amsmath}
\usepackage{textcomp}
\usepackage{amssymb,amsthm}
\usepackage[all]{xy}
\usepackage{color}
\usepackage{float}
\usepackage{enumerate}

\theoremstyle{definition}

\newtheorem{example}{Example}
\newtheorem{alphatheorem}{Theorem}
\newtheorem{lemma}{Lemma}
\newtheorem{remark}{Remark}
\newtheorem{corollary}{Corollary}
\newtheorem{prop}{Proposition}
{\theoremstyle{remark}

}
\newtheorem{definition}{Definition}
\newtheorem{alphacorollary}[alphatheorem]{Corollary}
\newtheorem{conjecture}{Conjecture}

\newcommand{\mm}[4]{\left(\begin{array}{cc} #1 & #2 \\ #3 & #4 \end{array}\right)}
\newcommand{\Z}{\mathbb{Z}}
\newcommand{\GL}{\mathrm{GL}}
\newcommand{\SL}{\mathrm{SL}}
\newcommand{\SU}{\mathrm{SU}}
\newcommand{\orb}{\mathsf{Orb}}
\newcommand{\stab}{\mathsf{Stab}}

\newcommand{\C}{\mathbb{C}}
\newcommand{\quot}{/\!\!/}
\newcommand{\X}{\mathfrak{X}}
\newcommand{\QS}{\mathfrak{Q}}

\newcommand{\frk}{\!/\!}
\newcommand{\F}{\mathbb{F}}
\newcommand{\strata}{\mathfrak{s}}
\newcommand{\Dbar}{{\overline{D}}}
\newcommand{\fr}{\mathsf{F}}
\newcommand{\R}{\mathfrak{R}}
\newcommand{\hm}{\mathrm{Hom}}
\newcommand{\id}{\mathbb{I}}
\newcommand{\NR}{N\!\!R}
\newcommand{\AI}{A\!I}
\newcommand{\Ab}{A\!b}
\newcommand{\tr}{\mathrm{tr}}

\newcommand{\sn}{\mathfrak{s}_{\text{\tiny N\!R}}}

\title[E-Polynomial of Character Varieties]{E-polynomial of $\SL_2(\C)$-Character Varieties of Free groups}

\date{\today}

\author[S. Cavazos]{Samuel Cavazos}

\address{Mathematics Department, Northwestern University, 2033 Sheridan Road Evanston, IL 60208-2730, USA}

\email{cavazos@math.northwestern.edu}

\author[S. Lawton]{Sean Lawton}

\address{Department of Mathematics, The University of Texas-Pan American,
1201 West University Drive Edinburg, TX 78539, USA}

\email{lawtonsd@utpa.edu}

\subjclass[2010]{14L30, 14D20, 14G05, 14G15}

\keywords{Free group, conjugacy class, character variety, finite field, E-polynomial}

\thanks{The second author was partially supported by the project GEAR (NSF-RNMS 1107367) USA, Simon's Foundation Collaboration grant (\#245642), and NSF grant 1309376.  The first author was supported by Undergraduate Research Initiative at the University of Texas-Pan American, and the Louis Stokes Alliances for Minority Participation (LSAMP)} 

\dedicatory{This paper is dedicated to Adalyn Belle Cavazos.}

\makeatother

\begin{document}
\begin{abstract}
Let $\fr_r$ be a free group of rank $r$, $\F_q$ a finite field of order $q$, and let $\SL_n(\F_q)$ act on $\hm(\fr_r,\SL_n(\F_q))$ by conjugation.  We describe a general algorithm to determine the cardinality of the set of orbits $\hm(\fr_r,\SL_n(\F_q))/\SL_n(\F_q)$.  Our first main theorem is the implementation of this algorithm in the case $n=2$.  As an application, we determine the $E$-polynomial of the character variety $\hm(\fr_r,\SL_2(\C))\quot \SL_2(\C)$, and of its smooth and singular locus.  Thus we determine the Euler characteristic of these spaces.
\end{abstract}

\maketitle

\section{Introduction}
In recent years there has been many new results concerning the $E$-polynomial of twisted character varieties:  \cite{HaRo}, \cite{LE}, \cite{MM}, \cite{LoMu} and \cite{LoMuNe}.  In this paper we consider free group character varieties.

Let $G$ be a reductive algebraic group over an algebraically closed field $\mathbb{F}$, and let $\Gamma$ be a finitely generated group.  Let $G$ act on $\hm(\Gamma,G)$ by conjugation.  Then the ring of invariants $\mathbb{F}[\hm(\Gamma,G)]^G$ is finitely generated since $G$ is reductive and consequently we have the GIT quotient $$\X_\Gamma(G):=\hm(\Gamma,G)\quot G:=\mathrm{Spec}_{max}\left(\mathbb{F}[\hm(\Gamma,G)]^G\right).$$  These spaces are called {\it character varieties}, and are of central importance in differential geometry, deformation theory of geometric structures, and in mathematical physics (see \cite{Si4}, and references therein).

When $\Bbbk$ is a sub-field of $\mathbb{F}$, it is natural to ask about the $\Bbbk$-points in $\X_\Gamma(G)$.  As has been shown in \cite{HaRo}, this can lead, via the Weil Conjectures, to an understanding of the topology of $\X_\Gamma(G)$; in particular, the Euler characteristic.  In the case when $G=\SL_2(\C)$ and $\Gamma=\fr_r$ is a free group of rank $r$, the Euler characteristic is known by results in \cite{FlLa}.  The methods used in \cite{FlLa} do not seem easy to generalize since they are underpinned by equivariant cohomology results in \cite{Ba0} that themselves seem hard to generalize.  

Let $\F_q$ be a finite field of order $q$.  In this paper we first describe the orbit set and its cardinality, $\QS_r(\SL_2(\F_q)):=\hm(\fr_r,\SL_2(\F_q))/ \SL_2(\F_q)$, and use it to determine the $E$-polynomial of the GIT quotient $\X_{\fr_r}\left(\SL_2(\C)\right)$, and that of the free Abelian case $\X_{\mathbb{Z}^r}\left(\SL_2(\C)\right)$ as well. 

Our first main theorem is the following:
\begin{alphatheorem}\label{theorema}  Let $r\geq 2$ and $q$ be odd.
The cardinality of $\mathfrak{Q}_r(\SL_2(\F_q))$ is

$$\mathcal{C}_r(q)=\frac{(q-3)(q-1)^{r-1}}{2}+\frac{(q-1)(q+1)^{r-1}}{2}+2^{r+1}q^{r-1}+2(q^3-q)^{r-1}.$$
\end{alphatheorem}

Since the conjugation action is not free, counting the orbits is not direct.  Consequently, we stratify the set $\hm(\fr_r,\SL_2(\F_q))$ into orbit-types that allows us to use a generalization of the classical group theory theorem of Lagrange to count the orbits in each stratum.  We first determine how many strata there are, and then describe them.  This provides a detailed description of the Diophantine geometry of $\hm(\fr_r,\SL_2(\F_q))$ and $\mathfrak{Q}_r(\SL_2(\F_q))$. Then we determine the cardinality of each stratum and the cardinality of its uniform stablilizer.  Using this we prove the theorem.

Thereafter, in Section \ref{sectionproof}, we prove our main application using a theorem of Katz (see Appendix of \cite{HaRo}). Let $X^{sm}$ (respectively $X^{sing}$) denote the smooth points (respectively singular points) of a variety $X$.

Here is our second main theorem:
\begin{alphatheorem}\label{epolyniomial}Let $q=xy$.  Then the $E$-polynomial for $\X_{\fr_r}(\SL_2(\C))$ is 
$$E_{\fr_r}(q)=(q-1)^{r-1}
   \left((q+1)^{r-1}-1\right)
   q^{r-1}+\frac{1}{2} q
   \left((q-1)^{r-1}+(q+1)^{r-1}
   \right),$$
and the $E$-polynomial of $\X_{\fr_r}(\SL_2(\C))^{sing}\cong\X_{\Z^r}(\SL_2(\C))$ is given by $$E_{\Z^r}(q)=\frac{1}{2} \left((q-1)^r+(q+1)^r\right).$$  Consequently, the difference of these is the $E$-polynomial of $\X_{\fr_r}(\SL_2(\C))^{sm}$.
\end{alphatheorem}

We note that the reducible (or Abelian) strata corresponds to the singular locus by \cite{FlLa2} for $r\geq 3$ (for $r=1,2$ the moduli space is smooth), and so this is how we recover the $E$-polynomial of $\X_{\mathbb{Z}^r}(\SL_2(\C))$ and $\X_{\fr_r}(\SL_2(\C))^{sing}$.  

To simplify the notation in what follows, we will often shorten $\X_{\fr_r}(\SL_2(\C))$ to $\X_r$.   Let $\chi(X)$ denote the Euler characteristic of a topological space $X$.

\begin{alphacorollary}\label{eulercor}
$\chi(\X_r)=2^{r-2}$, $\chi(\X_r^{sm})=-2^{r-2}$, $\chi(\X_r^{sing})=2^{r-1}$, $\chi((\X_r^{sing})^{sm})=-2^{r-1}$, and $\chi((\X_r^{sing})^{sing})=2^{r}$.
\end{alphacorollary}

In Section \ref{finalsection}, we give an independent proof of Corollary \ref{eulercor} by computing the Poincar\'e polynomials for these two moduli spaces.  Although to prove Corollary \ref{eulercor} it suffices to evaluate the $E$-polynomials in Theorem \ref{epolyniomial} at $q=1$, and use its additivity property for the other strata.

\section{E-Polynomial}
In what follows, for an affine variety $X$, we consider singular cohomology $H^*(X;\Bbbk)$ where $\Bbbk$ is a field of characteristic 0 (sometimes writing simply $H^*(X)$ when $\Bbbk$ is not important).  It is equivalent to simplicial cohomology since algebraic sets are simplicial, and also to sheaf cohomology with the constant sheaf since algebraic sets are locally contractible.

P. Deligne in \cite{De71, De74}  showed that a complex variety $X$ admits an increasing weight filtration $0=W_{-1}\subset W_0\subset\cdots\subset W_{2j}=H^j(X;\mathbb{Q})$, and a decreasing Hodge filtration $H^j(X;\mathbb{C})=F^{0}\supset\cdots\supset F^{m+1}=0$ such that for all $0\leq p\leq l$, $$\mathrm{Gr}^{W\otimes \mathbb{C}}_l:=W_l\otimes\mathbb{C}/W_{l-1}\otimes\mathbb{C}=
F^p(\mathrm{Gr}^{W\otimes \mathbb{C}}_l) \oplus\overline{F^{l-p+1}(\mathrm{Gr}^{W\otimes \mathbb{C}}_l)},$$ where $F^p(\mathrm{Gr}^{W\otimes \mathbb{C}}_l)=(F^p\cap W_l\otimes \mathbb{C}+W_{l-1}\otimes \mathbb{C})/W_{l-1}\otimes \mathbb{C}$.

This allows one to define the mixed Hodge numbers for every $H^j(X;\mathbb{C})$ by 
\begin{eqnarray*}
h^{p,q;j}(X)&:=&\mathrm{dim}_{\mathbb{C}}\mathrm{Gr}^F_p\left(\mathrm{Gr}^{W\otimes \C}_{p+q}H^j(X)\right)\\
&=&\mathrm{dim}_{\C}F^p(\mathrm{Gr}^{W\otimes \C}_{p+q})/F^{p+1}(\mathrm{Gr}^{W\otimes \C}_{p+q})\\
&=&\mathrm{dim}_{\C}F^p\cap(W_{p+q}\otimes \C)/(F^{p+1}\cap W_{p+q}\otimes \C +W_{p+q-1}\otimes \C\cap F^p),\end{eqnarray*} and subsequently define the mixed Hodge polynomial $$H(X;x,y,t):=\sum h^{p,q;j}(X)x^py^qt^j.$$

Likewise, one can also consider cohomology with compact support and obtain the same structure.  We denote this by $H_c^*(X;\Bbbk)$, and correspondingly the mixed Hodge numbers by $h^{p,q;j}_c$ and the mixed Hodge polynomial by $H_c(X;x,y,t)$.

The {\it $E$-polynomial} is defined to be $E(X;x,y):=H_c(X;x,y,-1)$.  This immediately implies that the classical Euler characteristic is given by $\chi(X)=E(X;1,1)$.  For further details, see \cite{PeSt08}.

A {\it spreading out} of $X$ is a scheme $\mathcal{X}$ over a $\Z$-algebra $R$ with an inclusion $\varphi:R\hookrightarrow \C$ such that the extension of scalars satisfies $\mathcal{X}_\varphi\cong X$.  $X$ is said to have {\it polynomial count} if there exists $P_X\in \Z[t]$ and a spreading out $\mathcal{X}$ such that for all homomorphims $\phi:R\to \F_q$ to finite fields (for all but finitely many primes $p$ so $q=p^k$) we have $\# \mathcal{X}_\phi(\F_q)=P_X(q)$.  Katz shows in \cite{HaRo} that if $X$ has polynomial count, then $E(X;x,y)=P_X(xy)$.

Let $\fr_r$ be a rank $r$ free group.  Then the representation variety $\hm(\fr_r,\SL_2(\C))$ is acted upon by $\SL_2(\C)$ via conjugation.  Let $\C[\hm(\fr_r,\SL_2(\C))]$ be the coordinate ring of the representation variety.  The GIT quotient $$\X_r=\mathrm{Spec}\left(\C[\hm(\fr_r,\SL_2(\C))]^{\SL_2(\C)}\right),$$ where $\C[\hm(\fr_r,\SL_2(\C))]^{\SL_2(\C)}$ is the ring of invariants, is called the $\SL_2(\C)$-character variety of $\fr_r$.  By Seshadri's extension of GIT to arbitrary base, see \cite{Se77}, there exists a scheme $\mathcal{X}_r=\mathrm{Spec}\left(R[\hm(\fr_r,\SL_2(R))]^{\SL_2(R)}\right)$, where $R=\Z[1/2]$.  Then since $R\hookrightarrow \C$ is a flat morphism, Lemma 2 in \cite{Se77} implies $$R[\hm(\fr_r,\SL_2(R))]^{\SL_2(R)}\otimes_R\C=\C[\hm(\fr_r,\SL_2(\C))]^{\SL_2(\C)}$$ and thus $\X_r$ admits a spreading out.

\section{Counting Representations}
In this section we introduce the notation and begin a discussion of the computations to prove the main theorems. 

\subsection{Notation}

Let $p$ be a prime integer, and let $\F_q$ be the finite field of order $q=p^k$ for $k\geq 1$.  The group $\GL_n(\F_p)$ is the group of $n\times n$ invertible matrices over $\F_q$, and $\SL_n(\F_q)$ is the subgroup of those elements in $\GL_n(\F_q)$ whose determinant is $1$.  Denote $\F_q^{\times n}-\{(0,0,...,0)\}$ by $(\F_q^{\times n})^*$.  In general, we will denote $\Bbbk^*=\Bbbk-\{0\}$ for any field $\Bbbk$. 

We begin this section by counting the total number of elements in $\hm(\fr_r,\SL_n(\F_q))$ where $\fr_r$ is a rank $r$ free group.  Let $|X|$ denote the cardinality of a set $X$.  Since $\hm(\fr_r, G)\cong G^{\times r}$ for any group $G$, and $|X\times Y|=|X||Y|$ for any sets $X$ and $Y$, it suffice to compute the order of $\SL_n(\F_q)$. 

For any group $G$ acting on a set $X$ let $\orb_G(x)$ denote the orbit of $x\in X$, and $\stab_G(x)$ denote its stabilizer.  When $G$ and $X$ are finite, the Orbit-Stabilizer Theorem tells $|G|=|\orb_G(x)||\stab_G(x)|$ (see \cite{H}).

\subsection{Representations}

\begin{lemma}
\label{grouplemma}
${\displaystyle |\GL_n(\F_q) |=q^{\frac{n(n-1)}{2}}\prod_{k=1}^{n}(q^{k}-1)}$ and ${\displaystyle |\SL_n(\F_q)|=q^{\frac{n(n-1)}{2}}\prod_{k=2}^{n}(q^k-1).}$
\end{lemma}

\begin{proof} 
Take a matrix in $\GL_n(\F_q)$.  The first column is non-zero and thus there are $q^n-1$ choices.  The second column must be linearly independent from the first, so there are $q^n-q$ choices.  Likewise, each subsequent column must be linearly independent from the previous columns, so $|\GL_n(\F_q) |=\prod_{k=0}^{n-1}(q^n-q^k)$, and this simplifies to $q^{\frac{n(n-1)}{2}}\prod_{k=1}^{n}(q^{k}-1)$.  The determinant homomorphim $\GL_n(\F_q)\to \F_q^*$ is onto with kernel $\SL_n(\F_q)$.  Thus, set theoretically $\GL_n(\F_q)$ is $\F_q^*\times \SL_n(\F_q)$, and $|\SL_n(\F_q)|$ equals $|\GL_n(\F_q)|/(q-1)=q^{\frac{n(n-1)}{2}}\prod_{k=2}^{n}(q^k-1).$
\end{proof}

\begin{corollary}\label{sec:corollary13}
$$|\hm(\fr_r,\GL_n(\F_q))|=\left( q^{\frac{n(n-1)}{2}}\prod_{k=1}^n(q^k-1)\right)^r$$ and $$|\hm(\fr_r,\SL_n(\F_q))|=\left( q^{\frac{n(n-1)}{2}}\prod_{k=2}^n(q^k-1)\right)^r.$$
\end{corollary}

\begin{example}
The cardinality of the sets above coincides with the number of $\F_q$-points in the $\Z$-schemes $\mathit{Hom}(\fr_r, \GL_n)$ and $\mathit{Hom}(\fr_r,\SL_n)$ since they are products of group schemes, and the cardinality of these groups by definition corresponds to the $\F_q$-points of the associated schemes.  Thus these varieties are of type polynomial-count, and so the counting polynomials are the $E$-polynomials by Katz's work in the appendix of \cite{HaRo} (see the previous section for definitions and references).  Consequently, the Euler characteristic of the space of $\C$-points is $0$ by setting $q=1$.  This is as expected since $\chi(\hm(\fr_r, \SL_n(\C)))=\chi(\SU(n))^r=0$ given that $\SU(n)$ is a fibration over $S^{2n-1}$; and so $\chi(\hm(\fr_r, \GL_n(\C)))=\chi(\hm(\fr_r,\SL_n(\C)))\chi(\hm(\fr_r,\C^*))=0$.
\end{example}

\begin{remark}\label{splitremark}
All of the above computations can be generalized greatly, in fact any split reductive algebraic group $G$ is polynomial count.  Moreover, using the Bruhat Decomposition, as shown by Chevalley in \cite{Ch}, the explicit counting-polynomial can be written.  Thus, $\hm(\fr_r,G)$ is polynomial-count for any such $G$.
\end{remark}

\subsection{Characters}
Let $g\in\SL_n(\F_q)$ and $\rho\in\hm(\fr_r,\SL_n(\F_q))$. Then $\SL_n(\F_q)$ acts by conjugation on $\hm(\fr_r,\SL_n(\F_q))$; $g\cdot\rho=g\rho g^{-1}$.  Through the evaluation mapping identifying $\hm(\fr_r,\SL_n(\F_q))$ with $\SL_n(\F_q)^{\times r}$ this action becomes simultaneous conjugation; $g\cdot (g_1,...,g_r)=(gg_1g^{-1},...,gg_rg^{-1}).$

We can therefore formulate the quotient space $$\QS_r(\SL_n(\F_q))=\hm(\fr_r, \SL_n(\F_q))/\SL_n(\F_q),$$ which is by definition the set of conjugation orbits of homomorphisms.  Our first goal in the coming sections is to determine $|\QS_r(\SL_2(\F_q))|$.  If the action were free, we would simply take the computation for $|\hm(\fr_r,\SL_n(\F_q))|$ and divide it by the computation of $|\SL_n(\F_q)|$.  However, the action is not free, so we will have to partition the set of homomorphisms into subsets of equal stabilizer type, whose quotients we will be able to count.  The strategy is then to relate $\QS_r(\SL_n(\F_q))$ to the $\F_q$-points of the $\Z[1/n]$-scheme associated to the character variety $\X_r(\SL_n(\C))=\hm(\fr_r,\SL_n(\C))\quot \SL_n(\C)$.  See \cite{Si2} for a detailed description of this scheme.  We will do this only for the case $n=2$, although we expect the $n=3$ case to likewise be tractable.

\section{Stratification}
In this section we divide $\hm(\fr_r,\SL_n(\F_q))$ into conjugate invariant subsets.  We choose such a stratification so that two homomorphisms are in the same stratum if and only if their stabilizers have the same cardinality. 

\begin{definition}
\label{uniformactiondef}
Let $G$ be a finite group acting on a set $X$. We say that two elements $x,y\in X$ have the same stabilizer type if $|\stab_G(x)|=|\stab_G(y)|$.  Then, $G$ is said to act uniformly on $X$ if there is exactly one stabilizer type for all $x\in X$. In this case, letting the cardinality of the stabilizer be $m$, we say that $G$ acts uniformly of order $m$ on $X$.
\end{definition}

Since $\hm(\fr_r,\SL_n(\F_q))$ is finite, there are a finite number of stabilizer types. Let $N$ be that number ($N$ may depend on $q,r,n$).  Denote $\strata_i\subset\hm(\fr_r,\SL_n(\F_q))$ for $1\leq i\leq N$ be the distinct and disjoint subsets of fixed stabilizer type. Then $\hm(\fr_r,\SL_n(\F_q))=\bigsqcup_{i=1}^N\strata_i.$  Consequently, $\QS_r(\SL_n(\F_q))=\bigsqcup_{i=1}^N(\strata_i/\SL_n(\F_q))$.  To count the total number of orbits in each stratum, we will make extensive use of the following proposition; itself a generalization of the Lagrange's Theorem (see \cite{H}).

\begin{prop}[Uniform Action Theorem] 
\label{uniformactionthm}
Let $X$ be a finite set, and let $G$ be a finite group acting uniformly of order $m$ on $X$. Then,
$|X/G|=m|X|/|G|. $
\end{prop}

\begin{proof}
This follows from Burnside's Counting Theorem (see \cite{H}), but we prove it here to be complete.  Since
$X=\bigsqcup_{i=1}^{|X/G|}\orb_G(x_i),$ the Orbit-Stabilizer Theorem implies $|X|=|X/G||G|/m,$ as required.
\end{proof}

By definition, $\SL_n(\F_q)$ acts uniformly of some order $m_i$ on $\strata_i$, for $1\leq i\leq N$.  Thus,
$$|\QS_r(\SL_n(\F_q))|=\sum_{i=1}^N\frac{m_i|\strata_i|}{q^{\frac{n(n-1)}{2}}\prod_{k=2}^n(q^k-1)}. $$

We now specialize to $n=2$, $q$ odd, and $r\geq 2$.  The special case of $r=1$, the character variety is $\C$ and so its E-polynomial is $q$.  Excluding $p=2$ does not change our results since we need only have a counting function that works on a dense set of primes.

\subsection{Strata}
For the rest of this paper, let $G=\SL_2(\F_q)$ and denote $\hm(\fr_r, G)$ by $\R_r$.  We also assume that $p$ is odd and $r\geq 2$.  Let $\mathbb{I}$ be the identity matrix.  Denote the algebraic closure of $\F_q$ by $\overline{\F}_q$. 

Let $Z$ be the center of $G$.  It is easy to see that $Z=\{\pm \mathbb{I}\},$ and so has order 2 if and only if the characteristic of $\F_q$ is odd.

\begin{definition}\label{stratadef}
\begin{enumerate}
\item[]

\item[$1.$]  Define $\strata_Z$ to be $\{\rho\in \R_r\ |\ \rho(w)\in Z,\text{ for all } w\in \fr_r\}.$  We call these homomorphisms central, and this set the central stratum.

\item[$2.$] Let $D$ denote the set of diagonal matrices in $G$.  Define $\strata_D$ to be $\{\rho\in \R_r\ |\text{ there exists } g\in G\text{ such that } g\rho(w)g^{-1}\in D,\text{ for all } w\in \fr_r\}-\strata_Z.$  We call these homomorphisms diagonalizable, and this set the diagonalizable stratum.

\item[$3.$] Let $\overline{D}$ denote the set of diagonal matrices in $\SL_2(\overline{\F}_q)$.  Define $\strata_{\overline{D}}$ to be $\{\rho\in \R_r\ |\text{ there exists } g\in \SL_2(\overline{\F}_q)\text{ such that } g\rho(w)g^{-1}\in \overline{D},\text{ for all } w\in \fr_r\}-\strata_D\cup\strata_Z.$  We call these homomorphisms extendably diagonalizable, and this set the extendably diagonalizable stratum.

\item[$4.$] Let $U=\left\{\mm{\pm 1}{a}{0}{\pm1} | \ a \in\F_q \right\}.$  Define $\strata_U$ to be
$\{\rho\in \R_r\ |\text{ there exists } g\in G \text{ such that } g\rho(w)g^{-1}\in U,\text{ for all } w\in \fr_r\}-\strata_{\overline{D}}\cup\strata_D\cup\strata_Z.$  We call these homomorphisms projectively unipotent, and this set the projectively unipotent stratum.

\item[$5.$] Define $\strata_{N}$ to be $\R_r-\strata_U\cup\strata_{\overline{D}}\cup\strata_D\cup\strata_Z.$  We call these homomorphisms non-Abelian, and this set the non-Abelian stratum.  
\end{enumerate}
\end{definition}

Call an element of $\rho\in\R_r$ {\it Abelian} if its image is an Abelian group, and call $\rho$ {\it reducible over} $H\subset \SL_2(\overline{\F}_q)$ if there exists $g\in H$ so that $g\rho g^{-1}$ has its image contained in the set of upper-triangular matrices.  When $H$ is not specified, we mean reducible over $\SL_2(\overline{\F}_q)$. Representations will be called {\it absolutely irreducible} if $\rho$ is not reducible over $\SL_2(\overline{\F}_q)$.  We will see that $\strata_{N}$ consists of absolutely irreducible homomorphims, and also reducible homomorphisms that are not projectively unipotent, not (extendably) diagonalizable, and not central.

\begin{prop}\label{strataprop}
The sets defined in Definition \ref{stratadef} are disjoint conjugate invariant sets whose union equals $\R_r$.  Moreover,  (a) $\strata_{N}$ is exactly the set of non-Abelian homomorphisms which consists of absolutely irreducible homomorphims and non-Abelian reducible homomorphims, and (b) $G$ acts on each stratum uniformly.
\end{prop}

\begin{proof}
By definition all the sets are disjoint from each other, are conjugate invariant, and $\R_r(G)=\strata_Z\cup\strata_D\cup\strata_{\overline{D}}\cup\strata_U\cup\strata_{N}$.   It is not hard to see that $\strata_Z\cup\strata_D\cup\strata_{\overline{D}}\cup\strata_U$ are Abelian and reducible over $\SL_2(\overline{\F}_q)$, and $\strata_{N}$ is by definition their complement.  Thus, $\strata_N$ contains all absolutely irreducible homomorphims (they are non-Abelian since by Burnside's Theorem (\cite{Lang}, p.649) they algebraically generate all $2\times 2$ matrices but if they were Abelian they would generate an Abelian algebra which is a strict subset).  By Shur's Lemma (see \cite{J}) the stabilizer of the absolutely irreducible representation must be the center of $G$, and so $G$ acts uniformly on that subset of $\strata_N$.  So to prove (a) we need to show that $\strata_{N}$ does not contain any Abelian homomorphisms (therefore Abelian homomorphisms are necessarily reducible over $\overline{\F}_q$).  This 
will follow from part (3) of the No-Mixing Theorem below.  We will prove the rest of part (b) in the next section.
\end{proof}

\begin{definition}
Let the non-Abelian reducible representations be denoted by $\strata_{\NR}$ (so $\strata_{\NR}\subset \strata_{N}$), and let the absolutely irreducible representations be denoted by $\strata_{\AI}$.  Thus, $\strata_{\AI}=\strata_{N}-\strata_{\NR}$.  Lastly, denote the Abelian representations by $\strata_{\Ab}$; so $\strata_{\Ab}=\strata_Z\cup\strata_D\cup\strata_{\overline{D}}\cup\strata_U$.
\end{definition}

\begin{remark}
The above proposition also shows that the number of strata for $q$ odd and $r\geq 2$ does not depend on $q$ or $r$.  In particular, excepting $r=1$ or $q=2^k$, there are always 5 strata.  We conjecture that for $\QS_r(\SL_n(\F_q))$ the number of strata $N$, for $r\geq 2$ and $q$ such that $\gcd(q,n)=1$, depends only on $n$.
\end{remark}

In analogy with the usual notion of quadratic residues over $\Z_p$, we will call an element $a\in \F_q^*$ a quadratic residue if there exists a solution in $\F_q$ to the equation $x^2=a$.  Otherwise, $a$ is called a quadratic non-residue.  We will also use the Legendre symbol $\left(\frac{a}{q}\right)$ to be $1$ if $a$ is a residue and $-1$ otherwise.

\begin{lemma}\label{residuelemma}
Let $\F_q$ be a finite field of order $q=p^k$ where $p$ is an odd prime.  Then there exists $(q-1)/2$ residues in $\F_q^*$.  Moreover, for all $a,b\in \F_q^*$, $\left(\frac{a}{q}\right)\left(\frac{b}{q}\right)=\left(\frac{ab}{q}\right)$. 
\end{lemma}

\begin{proof}
It is just the observation that the $\F_q^*$ is cyclic of even order and so isomorphic to $\Z_{2m}$, which has exactly half its elements multiples of $2$.  Thus in $\F_q^*$ there are $(q-1)/2$ squares.  The set of these squares is a subgroup $S$, so $\F_q^*/S\cong \Z_2$.  Thus, a non-residue times a non-residue must be a residue.  From the definition of residue alone, a residue times a residue is a residue, and a residue times a non-residue is a non-residue.  The result follows.
\end{proof}

\begin{remark}\label{uniqueextension}
Any element $a\in \F_q$ which is a quadratic non-residue produces a quadratic extension $\F_q(\sqrt{a})=\F_q[x]/(x^2-a)$, and by the proof of the above lemma $\F_q(\sqrt{a}) = \F_q(\sqrt{b})$ for any quadratic non-residues $a,b$ (since $\sqrt{a}\sqrt{b^{-1}}\in \F_q$).  Thus, there is a unique quadratic extension of $\F_q$; which we denote by $\F_{q^2}$.
\end{remark}

We will denote column vectors $(x,y)^\dagger$, where the symbol $\dagger$ means transpose.

\begin{lemma}
\label{eigenlemma}
Let $M\in\SL_2(\F_q)$, let $v=( v_1, v_2)^\dagger$ be an eigenvector of $M$ with $e$ its eigenvalue.  Suppose $e\in \F_{q^m}-\F_q$ for some non-trivial field extension $\F_{q^m}/\F_q$. Then $v\notin \F_q^{\times 2}$.   Moreover, $e$ is in the unique quadratic extension $\F_{q^2}$.
\end{lemma}

\begin{proof}
Suppose that $v\in\F_q^{\times 2}$. Then $Mv=ev\in\F_q^{\times 2}$.  Without loss of generality let $v_1\neq0$. Since $ev_1:=\ell\in\F_q^*$, then $e=v_1^{-1}\ell\in \F_q$, which contradicts that $e\in\F_{q^m}-\F_q$.

Since $e$ is a zero of the characteristic polynomial $x^2-tx+1$ where $t=\tr(M)\in \F_q$, $e=2^{-1}(t\pm \sqrt{t^2-4})\in\F_q(\sqrt{t^2-4})$.  Since $e\notin \F_q$, $t^2-4$ is a quadratic non-residue and Remark \ref{uniqueextension} implies that $e$ is in the unique extension $\F_{q^2}$.  

\end{proof}

\begin{remark}
With respect to Lemma \ref{eigenlemma}, since the scalar of an eigenvector is again an eigenvector, we cannot make any further conclusions about where the coordinates of $v$ lie.  However, $v_1$ and $v_2$ are both non-zero since at least one must be, and if the other was zero, then by scaling the non-zero coordinate to 1, we contradict $e\notin\F_q$.  Then, scaling $v$ by $1/v_1$ we obtain the eigenvector $(1,w)^\dagger$ where $w=v_2/v_1\notin\F_q$.  However then, denoting the coordinates of $M$ by $m_{ij}$, we have $m_{11}+m_{12}w=e$, which implies that $m_{12}\not=0$ and so $w=(e-m_{11})/m_{12}\in \F_{q^2}$.  Thus, if either $v_1$ or $v_2$ is in $\F_{q^2}$, then the other is in $\F_{q^2}$ too.
\end{remark}

\begin{lemma}\label{basefieldreductionlemma}
Let $\rho=(A_1,...,A_r)\in \hm(\fr_r,G)$ and suppose that for each $i$ the eigenvalues of $A_i$ are in $\F_q$, and for at least one $A_i$ its eigenvalues are not repeated.   If $\rho$ is upper-triangularizable over $\overline{\F}_q$, then $\rho$ is upper-triangularizable over $\F_q$.  If $\rho$ is diagonalizable over $\overline{\F}_q$, then $\rho$ is diagonalizable over $\F_q$. 
\end{lemma}

\begin{proof}
By assumption, there exists $g\in\SL_2(\overline{\F}_q)$ so that $g^{-1}\rho g$ is upper-triangular.  Since at least one such matrix, say $g^{-1}A_ig=\mm{x}{y}{0}{x^{-1}}$ with $x\in \F_q^*$ and $y\in\overline{\F}_q$, has non-repeated eigenvalues, we can conjugate $g^{-1} \rho g$ further by $\left(\begin{array}{cc}1& \frac{y}{x-x^{-1}}\\0&1\end{array}\right)$ to make $A_i$ diagonal yet keep $g^{-1}\rho g$ upper-triangular.  We assume we have done so without changing notation; in particular, $g^{-1}A_ig=\mm{x}{0}{0}{x^{-1}}$.   Now denote $A_i=\mm{a}{b}{c}{d}$, the component matrix before conjugation.  Assuming $g$ has the form $\mm{v_1}{u_1}{v_2}{u_2}$, then $v=(v_1, v_2)^\dagger$ is an eigenvector of $A_i$ with eigenvalue $x$, and $u=(u_1, u_2)^\dagger$ is an eigenvector of $A_i$ with eigenvalue $x^{-1}$; clearly $\{u,v\}$ are linearly independent.  Since $A_1,...,A_r$ are simultaneously upper-triangularizable, they share a common eigenvector, so some multiple of $v$ or $u$ is a simultaneous eigenvector for 
$A_1,...,A_
r$.   But every multiple of an eigenvector is again an eigenvector, so $v$ or $u$ is a common eigenvector.  If $\rho$ were diagonalizable, then both $v$ and $u$ would be simultaneous eigenvectors.

Now suppose $b$ and $c$ are both $0$, that is $A_i$ is diagonal with distinct eigenvalues, then since the only eigenvectors of $A_i$ are multiples of $e_1=(1, 0)^\dagger$ and $e_2=(0, 1)^\dagger$, $e_1$ or $e_2$ is a common eigenvector.  In the first case $\rho$ is upper-triangular to begin with, and in the second case it is lower-triangular to begin with.  Since conjugating $\rho$ by $f:=\left(
\begin{array}{cc}
 0 & 1 \\
 -1 & 0 \\
\end{array}
\right)$ turns a lower-triangular representation into an upper-triangular one, we conclude that in the case $b=0=c$, that $\rho$ is upper-triangularizable over the base field $\F_q$.  Moreover, the same argument shows that if $\rho$ was diagonalizable, the matrices $A_1,...,A_r$ would share two common eigenvectors, and so they would be have to be $e_1$ and $e_2$.  In that case we conclude that $\rho$ was diagonal to begin with.

If $b=0$ but $c\not =0$, then by conjugating $\rho$ by $f$ and replacing $g$ by $fg$ the $i$-th component $A_i$ becomes strictly upper-triangular and the $i$-th component of $g^{-1}\rho g$ remains diagonal.   So we can assume that $b\not=0$.

Then, we conclude that there exist $\lambda, \mu\in \overline{\F}_q$ such that $(v_1, v_2)^\dagger=\lambda(b, x-a)^\dagger$ and $(u_1, u_2)^\dagger=\mu(b, x^{-1}-a)^\dagger,$ and so $g=\mm{\lambda b}{\mu b}{\lambda(x-a)}{\mu(x^{-1}-a)}$.  However, as noted before, all multiples $\lambda, \mu$ provide simultaneous eigenvectors.  So in fact, we know that one of $v:=(b, x-a)^\dagger$ or $u:=(b, x^{-1}-a)^\dagger$ is a simultaneous eigenvector (both are if $\rho$ was diagonalizable).  To verify that these vectors are in fact eigenvectors for $A_i$, we check $g^{-1}A_ig$ in the case $\lambda=1=\mu$ to obtain $$\left(
\begin{array}{cc}
 -b (x (a+d)-a d+b c-1) &
   \frac{b x^2 (a d-b c)-b x
   (a+d)+b}{x^2} \\
 b ((a-x) (x-d)+b c) & \frac{b
   (a (-d) x+a+x (b
   c-1)+d)}{x} \\
\end{array}
\right).$$

The lower left entry and the upper right entry each simplify (using the identities $ad-bc=1$ and $a+d=x+x^{-1}$) to multiples of the characteristic polynomial in terms of $x$.  Hence they are each 0 since $x$ is an eigenvalue.
 
We claim that $C:=\mm{b}{b}{x-a}{x^{-1}-a}$, after possibly conjugating by $f$, will upper-triangulize $\rho$, and in fact will diagonalize $\rho$ if $\rho$ was diagonalizable.  Since this matrix is over the base field, we are done.

First note that $C$ is invertible since $\det(C)=0$ if and only if $b(x^{-1}-a-x+a)=0$, which only occurs if $b=0$ or $x=\pm 1$.  Neither is true given our assumptions at this point in the argument.  Moreover, we can scale the columns of $C$ over the base field (as described above), and preserve their being simultaneous eigenvectors over $\F_q$ and yet arrange for the $\det(C)=\lambda\mu b(x^{-1}-x)=1$; for instance $\lambda=1$ and $\mu=(b(x^{-1}-x))^{-1}$.

Indeed, suppose that $v$ is the simultaneous eigenvector with eigenvalue $x_j$ for the matrix factor $A_j$, then for each $A_j$, we have $$A_j[v|u]=[A_jv |A_ju]=[x_jv|A_ju]=[v|u]\mm{x_j}{*}{0}{*},$$ with second column $(*, *)^\dagger=C^{-1}A_jw$.  Thus, $C^{-1}\rho C$ is upper-triangular where $C$ is over the base field.  

If $u$ is the simultaneous eigenvector, then likewise we have $$A_j[v|u]=[A_jv |A_ju]=[A_jv|x_ju]=[v|u]\mm{*}{0}{*}{x_j},$$ with first column $(*, *)^\dagger=C^{-1}A_jv$.  Thus $C^{-1}\rho C$ is lower-triangular over the base field.  Conjugating by $f$ makes it upper-triangular over $\F_q$.

And if both $v$ and $u$ are simultaneous eigenvectors, then for each $A_j$ we have $$A_j[v|u]=[A_jv |A_ju]=[x_j v|x_j^{-1}u]=[v|u]\mm{x_j}{0}{0}{y_j},$$
and thus $C^{-1}\rho C$ is diagonal where again $C$ is over $\F_q$.
\end{proof}

Denote the free group on $r$ letters as $\fr_r=\langle \gamma_1,...,\gamma_r\rangle$.  The next proposition loosely says that if $\rho\in \hm(\fr_r,\SL_2(\F_q))$ is reducible, then $\rho(\gamma_i)$ for each $i$ either has all its non-trivial eigenvalues in the quadratic extension of $\F_q$ but not in $\F_q$ itself, or all of the eigenvalues are in $\F_q$--hence the name, ``no-mixing".

\begin{prop}[No-Mixing Theorem]\label{nomixing}  Let $q=p^k$ for $p$ an odd prime.
 Let $\rho\in \hm(\fr_r,\SL_2(\F_q))$ be reducible over $\overline{\F}_q$. Then:
 \begin{enumerate}
  \item For any fixed $1\leq i,j\leq r$ it is impossible for $\rho(\gamma_i)\neq \pm \mathbb{I}$ with eigenvalues in $\F_q$ and $\rho(\gamma_j)$ with an eigenvalue in $\overline{\F}_q-\F_q$.
  \item If there exists $1\leq i\leq r$ so $\rho(\gamma_i)$ has an eigenvalue in $\overline{\F}_q-\F_q$, then $\rho\in \strata_{\overline{D}}$.  Conversely, if $\rho\in \hm(\fr_r,\SL_2(\overline{\F}_q))$ is diagonalizable, and for all $1\leq i\leq r$ either $\rho(\gamma_i)$ has an eigenvalue of $\pm 1$, or $\rho(\gamma_i)$ has an eigenvalue not in $\F_q$ that satisfies $x^2-t_ix+1=0$ for $t_i\in \F_q$, then there exists $g\in \SL_2(\overline{\F}_q)$ so that $g\rho(\fr_r)g^{-1}\subset \SL_2(\F_q)$.
  \item If for all $1\leq i\leq r$, $\rho(\gamma_i)$ has eigenvalues in $\F_q$, then $\rho\not\in\strata_{\overline{D}}$, and $\rho$ is either in $\strata_Z\cup \strata_U\cup \strata_D$ or $\rho$ is non-Abelian (consequently $\strata_N$ contains only non-Abelian homomorphisms).
 \end{enumerate}
\end{prop}

\begin{proof}
Since $\rho$ is reducible, there is $h\in \SL_2(\overline{\F}_q)$ so $h\rho(\gamma_i) h^{-1}:=A_i$ are simultaneously upper-triangular.  Let $A_i=\left(\begin{array}{cc}a_i & x_i\\ 0 & a_i^{-1}\end{array}\right)$.  To simplify notation we now assume that $\rho$ is already in upper-triangular form.  Note that the characteristic polynomial says that $(a_i^{\pm 1})^2-t_i(a_i^{\pm 1})+1=0$ where $t_i\in \F_q$ is the trace of $\rho(\gamma_i)$.  Thus, since $p>2$, we can write $a_i^{\pm 1}=2^{-1}\left(t_i\pm \sqrt{t_i^2-4}\right)\in \F_q(\sqrt{t_i^2-4})$.  Note $\F_q(\sqrt{t_i^2-4})$ is a quadratic field extension if and only if $t_i^2-4$ is a quadratic non-residue which occurs if and only if $x^2-t_ix+1$ is an irreducible polynomial over $\F_q$.  In this case, we work in the quadratic field extension $\F_q(\sqrt{t_i^2-4})$.  Let $e_i:=\sqrt{t_i^2-4}$ for simplicity.  In analogy with the complex numbers, let $\Im(s+te_i)=t$ and $\Re(s+te_i)=s$, where $s,t\in\F_q$.  Note that $\Im$ and $\Re$ are $\F_q$-linear. 
 
We first prove item (1).  Take $A_i:=A=\left(\begin{array}{cc}a & x\\ 0 & a^{-1}\end{array}\right)$ with $a\in \F_q-\{0,\pm 1\}$ and $A_j:=B=\left(\begin{array}{cc}b & y\\ 0 & b^{-1}\end{array}\right)$ with $b\in \overline{\F}_q-\F_q$.  Since $a\not=b$, we can assume $i\not=j$.  

Since $b\not=\pm 1,0$, conjugating $\rho$ by $\left(\begin{array}{cc}1& \frac{y}{b-b^{-1}}\\0&1\end{array}\right)$ we can assume that $y=0$.  Note that this does not change the upper-triangular form of $\rho$.  We now assume we have done this so $B$ is diagonal.  If $x\not=0$ then there exists $\zeta\in\overline{\F}_q^*$ such that $x\zeta^2=1$.  Further conjugating $\rho$ by $\left(\begin{array}{cc}\zeta& 0\\0&\zeta^{-1}\end{array}\right)$ we can assume that $x=0$ or $x=1$.  Again, we note that this does not change the upper-triangular form of $\rho$, and again, we assume now we have done this.

According to our hypothesis, there exists $g\in \SL_2(\overline{\F}_q)$ so that $g\rho(\fr_r)g^{-1}\subset \SL_2(\F_q)$, and in particular $M:=gAg^{-1}$ and $N:=gBg^{-1}$ are both in $\SL_2(\F_q)$.  

Let $v$ and $w$ be the columns of $g$ with coordinates denoted by $v=(v_1,v_2)^\dagger$ and $w=(w_1,w_2)^\dagger$.  Since $N=gBg^{-1}$ and $B$ is diagonal, the columns of $g$ are eigenvectors for $N$ (with eigenvalues $b$ and $b^{-1}$).  From Lemma \ref{eigenlemma} and its proof, we can deduce that both $v$ and $w$ are necessarily in $\overline{\F}_q^{\times 2}-\F_q^{\times 2}$.  Moreover, by considering the upper left entry of $N=g\mm{b}{0}{0}{b^{-1}}g^{-1}=\mm{bv_1w_2-b^{-1}w_1v_2}{(b^{-1}-b)v_1w_1}{(b-b^{-1})v_2w_2}{b^{-1}v_1w_2-bw_1v_2},$ if any one of $v_1,v_2,w_1,w_2$ is zero, given that $v_1w_2-w_1v_2=\det(g)=1$, we conclude that $N$ is not in $\SL_2(\F_q)$. Thus all coordinates of $v,w$ are in fact non-zero.  Since $N$ has two distinct eigenspaces we can still say more.  There must exist $\lambda, \mu\in \overline{\F}_q^*$ such that $g=\mm{\lambda n_{12}}{\mu n_{12}}{\lambda(b-n_{11})}{\mu(b^{-1}-n_{11})}$ where $N=(n_{ij}).$  This follows by simply observing that the columns are in fact eigenvectors 
for $N$.  For instance, $N(n_{12},b-n_{11})^\dagger=(bn_{12}, n_{22}b-1)^\dagger=(bn_{12},b^2-bn_{11})^\dagger=b(n_{12},b-n_{11})^\dagger,$ since $b^2-(n_{11}+n_{22})b+1=0$.  Also note that $n_{12}\not=0$, since otherwise $N$ has an eigenvalue, namely $n_{22}$, in $\F_q$, which it does not.

We now show that $M=gAg^{-1}$ cannot be in $\SL_2(\F_q)$; which is a contradiction.  We compute $M=g\mm{a}{x}{0}{a^{-1}}g^{-1}$ which equals
\begin{equation*}\left(
\begin{array}{cc}
 \frac{\lambda  \left(a^2 \mu 
   \left(\frac{1}{b}-n_{11}\right)-(a x \lambda +\mu )
   \left(b-n_{11}\right)\right)
   n_{12}}{a} & \frac{\lambda  n_{12}^2
   \left(-\mu  a^2+x \lambda 
   a+\mu \right)}{a} \\
 -\frac{\lambda 
   \left(b-n_{11}\right)
   \left(-\mu  a^2+b^2 x \lambda
    a+\mu +b \left(\mu  a^2-x
   \lambda  a-\mu \right)
   n_{11}\right)}{a b} &
   \frac{\lambda  \left(a x
   \lambda  b^2-a^2 \mu 
   b^2+\left(\mu  a^2-x \lambda 
   a-\mu \right) n_{11} b+\mu
   \right) n_{12}}{a b} \\
\end{array}
\right),\end{equation*} where $x=0,1$.  Using the equation $\det(g)=\lambda  \mu  n_{12}(b^{-1}-b)=1$, the upper right entry of $M=(m_{ij})$ simplifies to $m_{12}=n_{12}(a^{-1}-a)(b^{-1}-b)^{-1}+n_{12}^2x\lambda^2.$  However, $(b^{-1}-b)^{-1}=-\sqrt{t^2-4}/(t^2-4)$ is not in $\F_q$ since by assumption $2b=t+\sqrt{t^2-4}$ is not in $\F_q$ where $t=n_{11}+n_{22}\in \F_q$.  Thus, since $n_{12}(a^{-1}-a)\in \F_q^*$, we conclude that the upper right entry of $M$ is not in $\F_q$ if $x=0$, or if $x=1$ and either $\lambda^2\in\F_q$ or $\lambda^2\in\overline{\F}_q-\F_q(\sqrt{t^2-4})$.  Consequently, $M$ is not in $\SL_2(\F_q)$ in these cases; the desired contradiction.  If $x=1$ and $\lambda^2\in \F_q(\sqrt{t^2-4})-\F_q$, then $0=\Im(m_{12})=-n_{12}(a^{-1}-a)/(t^2-4)+n_{12}^2\Im(\lambda^2)$ if and only if $\Im(\lambda^2)=\frac{a^{-1}-a}{n_{12}(t^2-4)}$.  So for $M$ to be in $\SL_2(\F_q)$ the latter condition must hold.  Using this, we simplify $m_{21}$, and solve for $\Im(m_{21})=0$.  In this way, after a fairly 
lengthy calculation, we obtain that $\Re(\lambda^2)=\frac{a^{-1}-a}{n_{12}(2n_{11}-t)}$.  Note that since $a\not=\pm 1$, it must be the case that $2n_{11}-t=n_{11}-n_{22}\not=0$.  Thereafter, we substitute these necessary values for $\Re(\lambda^2)$ and $\Im(\lambda^2)$ into $m_{22}$.  Simplifying, again after a lengthy calculation, we obtain that $\Im(m_{22})=\frac{a^{-1}-a}{2(2n_{11}-t)}$.  But this latter expression is never $0$ since $a\not=\pm 1$.  This last contradiction finishes the proof of (1).

We now prove item (2).  Again we assume the $\rho=(A_1,...,A_r)$ is upper-triangular.  By (1), for each $i$ either $A_i$ is a multiple of the identity matrix, or the eigenvalues of $A_i$ are in $\F_q[x]/(x^2-tx+1)-\F_q$ where $\mathrm{tr}(A_i)=t$. Suppose that $\rho$ is not diagonalizable over $\overline{\F}_q$, yet has eigenvalues outside of $\F_q$.  Therefore, similar to the proof of (1), we can assume there exists distinct indices $i,j$ so $A_i=\mm{a}{1}{0}{a^{-1}}$ and $A_j=\mm{b}{0}{0}{b^{-1}}$ where $a,b\in\overline{\F}_q-\F_q$. 

By hypothesis, $\rho$ may be conjugated to a representation in $\SL_2(\F_q)$, and so there exists a $g\in \SL_2(\overline{\F}_q)$ so that $g\rho g^{-1} \in \SL_2(\F_q)$.
As in the proof of (1), we know there exists $\mu,\lambda\in \overline{\F}_q$ such that $g=\mm{\lambda n_{12}}{\mu n_{12}}{\lambda(b-n_{11})}{\mu(b^{-1}-n_{11})}$ where $N=(n_{ij})=gA_jg^{-1}$, and $\det(g)=\lambda\mu n_{12} \left(b^{-1}-b\right)=1$. We show that $M=gA_ig^{-1}$ cannot be in $\SL_2(\F_q)$; which is a contradiction.  Similarly to the calculation in the proof of (1), $M=(m_{ij})$ equals
\begin{equation*}\left(
\begin{array}{cc}
\frac{\lambda  n_{12} \left(a^2\mu \left(\frac{1}{b}-n_{11}\right)-\left(b-n_{11}\right) (a\lambda +\mu )\right)}{a} &
\frac{\lambda  n_{12}^2\left(-\mu a^2 +a \lambda +\mu \right)}{a} \\-\frac{\lambda \left(b-n_{11}\right) \left(b n_{11} \left(a^2 \mu -a
\lambda -\mu \right)-\mu a^2
+a b^2 \lambda +\mu
   \right)}{a b} & \frac{\lambda
    n_{12} \left(-a^2 b^2 \mu +b
   n_{11} \left(a^2 \mu -a
   \lambda -\mu \right)+a b^2
   \lambda +\mu \right)}{a b} \\
\end{array}
\right).\end{equation*}

By assumption $a=s/2+f/2$ where $f=\sqrt{s^2-4}$ and $s^2-4$ is a quadratic non-residue, and $s=\mathrm{tr}(A_i)$.  Likewise, $b=t/2+e/2$ where $e=\sqrt{t^2-4}$ and $t^2-4$ is a quadratic non-residue, and $t=\mathrm{tr}(A_j)$.  Simplifying $m_{12}$ with these values we determine that $m_{12}=\frac{e f n_{12}}{t^2-4}+\lambda^2 n_{12}^2$.  Since $e^2$ and $f^2$ are quadratic non-residues, the fact that the Legendre symbol is multiplicative (by Lemma \ref{residuelemma}) implies that $(ef)^2$ is a quadratic residue and therefore, $ef$ is in $\F_q$.  We thus conclude that $\lambda^2$ must also be in $\F_q^*$.  With that acknowledged, we now likewise simplify $m_{11}$ obtaining $$\frac{e f \left(2n_{11}-t\right)}{2
\left(t^2-4\right)}-\frac{1}{2} e \lambda ^2 n_{12}+\frac{\lambda ^2 n_{12} \left(2 n_{11} t^2-8 n_{11}-t^3+4 t\right)}{2\left(t^2-4\right)}+\frac{s}{2}.$$  Again, since $ef\in \F_q$ and $\lambda^2,n_{12}\not=0$ but are in $\F_q$, we deduce that $m_{11}\not\in \F_q$, the desired contradiction.  Therefore, we have shown that if $\rho$ is reducible and any component matrix has an eigenvalue not in $\F_q$, then $\rho\in \strata_{\overline{D}}$.  

We now prove the converse.  Since $\rho$ is diagonalizable, we assume that $\rho=(D_1,...,D_r)$ where each $D_i=\mm{d_i}{0}{0}{d_i^{-1}}$  and $d_i = \pm 1$, or $d_i^2-t_i d_i+1=0$, $t_i\in \F_q$ and $d_i \in \overline{\F}_q-\F_q$.  If $d_i=\pm 1$ for all $i$, then the result holds trivially.  So we assume there exists $i_0$ so $d_{i_0}\notin\F_q$.  

When $d_i\not=\pm 1$, then $d_i^{\pm 1}=\frac{t_i \pm \sqrt{t_i^2-4}}{2}=t_i/2\pm e_i/2$ where $e_i^2=t_i^2-4$ is a quadratic non-residue.  Thus, $\tr(D_i)=d_i+d^{-1}_i=t_i\in \F_q$.  Moreover, it is easy to show that $d_i=d_{j}^{\pm 1}$ if and only if $t_i=t_j$. 

There exists an $g_{i_0}\in \SL_2(\overline{\F}_q)$ so that $B_{i_0}=g_{i_0}D_{i_0}g_{i_0}^{-1}\in \SL_2(\F_q)$.  For instance, letting $g_{i_0}=\mm{1}{1}{d_{i_0}}{d_{i_0}^{-1}}$, we see that $B_{i_0}:=g_{i_0}D_{i_0}g_{i_0}^{-1}=\mm{0}{1}{-1}{t_{i_0}}.$  We claim that $g_{i_0}D_ig_{i_0}^{-1}\in\SL_2(\F_q)$ for all $1\leq i \leq r$.  Supposing that is the case and letting $g:=\det(g_{i_0})^{-1/2}g_{i_0}\in \SL_2(\overline{\F}_q)$, we have $g\rho g^{-1}=(g_{i_0}D_1g_{i_0}^{-1},...,g_{i_0}D_rg_{i_0}^{-1})\in\SL_2(\F_q)^{\times r}$, as desired.

Indeed, if $d_i=\pm 1$, then $g_{i_0}D_{i}g_{i_0}^{-1}=D_i \in \SL_2(\F_q)$.  If $d_i=d_{i_0}^{\pm 1}$, then either $g_{i_0}D_{i}g_{i_0}^{-1}=B_{i_0}\in \SL_2(\F_q)$, or $g_{i_0}D_i g_{i_0}^{-1}=g_{i_0}D_{i_0}^{-1}g_{i_0}^{-1}=B_{i_0}^{-1}\in\SL_2(\F_q)$. 
 
Otherwise, for each $i$ let $g_{i}=\mm{1}{1}{d_{i}}{d_{i}^{-1}}$ so that $B_i:=g_iD_ig_i^{-1}=\mm{0}{1}{-1}{t_{i}}\in \SL_2(\F_q)$.  Thus, $D_i=g_i^{-1}B_ig_i$ and so $g_{i_0}D_ig_{i_0}^{-1}=(g_{i_0}g_i^{-1})B_i(g_{i_0}g_i^{-1})^{-1}$.  It suffices only to prove $g_{i_0}g_i^{-1}\in \SL_2(\F_q)$.  Indeed, simplifying we obtain $g_{i_0}g_i^{-1}=\left(
\begin{array}{cc}
 1 & 0 \\
 \frac{t_{i_0}}{2}-\frac{e_{i_0} t_i}{2
   e_i} & \frac{e_{i_0}}{e_i} \\
\end{array}
\right).$  However, since $e_{i}^2$ is a quadratic non-residue, we conclude that $1/e_{i}^2$ is also a quadratic non-residue.  Again, since the Legendre symbol is multiplicative (Lemma \ref{residuelemma}) and $e_{i_0}^2$ is also a quadratic non-residue, we conclude that $(\frac{e_{i_0}}{e_i})^2$ is quadratic residue which implies that $\frac{e_{i_0}}{e_i}\in \F_q$.  Thus, $g_{i_0}g_i^{-1}\in \SL_2(\F_q)$ as needed.

Now we prove (3).  By Lemma \ref{basefieldreductionlemma}, if $\rho=(A_1,...,A_r)\in \strata_{\overline{D}}$ then some $A_i$ has an eigenvalue in $\overline{\F}_q-\F_q$.  Thus, by assumption, $\rho\not\in\strata_{\overline{D}}$.

We now prove that if all eigenvalues are $\pm 1$, then $\rho$ is in $\strata_Z\cup \strata_U$.  Assume that $\rho=(A_1,...,A_r)$ is not central.   By definition, there exists $g\in\SL_2(\overline{\F}_q)$ so that $g^{-1}\rho g$ is upper-triangular.  We will show that $g$ can be chosen from $\SL_2(\F_q)$.  Denote $A_{i}=\mm{a_{i}}{b_{i}}{c_{i}}{d_{i}}\in \SL_2(\F_q)$.  Then we can assume that $(b_{i},\epsilon_i-a_{i})^\dagger$ is an eigenvector for any non-central $A_i$, where $\epsilon_i=\pm 1$, is its eigenvalue.  Note that if $b_i=0$, then either $A_i$ is central or its only eigenspace is spanned by $(0,1)^\dagger$.  Since all the $A_i$'s have a common eigenvector, $\rho$ much be lower-triangular to begin with; we can thus act with the matrix $\mm{0}{-1}{1}{0}$ to make $\rho$ upper-triangular over $\F_q$ (likewise if one non-central component has $c_i=0$ then $\rho$ is upper-triangular to begin with).  We now assume that either each $A_i$ is $\pm \id$ or $b_i\not=0$ (with at least one such upper-right 
component, say $b_{i_0}$, non-zero).  We claim that $g=\mm{b_{i_0}}{x}{\epsilon_{i_0}-a_{i_0}}{y}$ for any $x,y\in\F_q$ such that $\det(g)=1$, like $y=b_{i_0}^{-1}(1+x(\epsilon_{i_0}-a_{i+0}))$, will make $g^{-1}A_{i}g$ upper-triangular for all $1\leq i\leq r$.  Obviously this holds for any central $A_i$, so we need only show that this holds for non-central components.  Since $A_1,...,A_r$ share exactly one common eigenspace, and since $(b_{i_0},\pm1-a_{i_0})^\dagger$ is an eigenvector for $A_{i_0}$, there is $\lambda_{i}\not=0$ so for any non-central $A_i$ we have $\lambda_i(b_i,\epsilon_i-a_i)^\dagger=(b_{i_0},\epsilon_{i_0}-a_{i_0})^\dagger$.  Thus computing the lower-left component in $g^{-1}A_ig$, with this substitution made for the first column of $g$, we obtain $-b_i \lambda _i^2\left(\epsilon _i^2- \left(a_i+d_i\right)\epsilon _i+1\right)$ which is $0$ since $\epsilon_i$ is an eigenvalue.  Thus we have shown that we can upper-triangulize any reducible $\rho$ whose components all have eigenvalues $\pm 
1$ over the base field $\F_q$ (Lemma \ref{basefieldreductionlemma} shows that this fact generalizes to $\strata_{\NR}$).

Now we can assume, since $\rho$ is reducible, that each $A_i$ is already upper-triangular and at least one such factor does not have eigenvalues $\pm 1$.  Indeed, let $A_{i_1}=\left(\begin{array}{cc}a&x\\0&a^{-1}\end{array}\right)$ with $a\not=\pm 1$.  Since $a\not=\pm 1$, conjugating by $\left(\begin{array}{cc}1& \frac{x}{a-a^{-1}}\\0&1\end{array}\right)$ we can assume that $x=0$.  At this point, $\rho\in \strata_D$ or it is not.  If it is not, then there is $A_{i_2}=\left(\begin{array}{cc}b&y\\0&b^{-1}\end{array}\right)$ with $y\not=0$.  Notice that $A_{i_1}A_{i_2}-A_{i_2}A_{i_1}$ is the zero matrix if and only if $ay=y/a$, which itself occurs only if $y=0$ or $a=\pm 1$.  Neither holds by construction.  Therefore, $A_{i_1}$ and $A_{i_2}$ do not commute and thus $\rho(\fr_r)$ is non-Abelian.

\end{proof}

\section{Uniform Action on Strata}

The point of this section is to prove that $\SL_2(\F_q)$ acts uniformly on each stratum defined in Definition \ref{stratadef}.  This will finish the proof of Proposition \ref{strataprop}. 

The following elementary proposition will be used repeatedly.

\begin{prop}
\label{stabprop}Let $G$ and $\Gamma$ be groups, and let $G$ act on $\hm(\Gamma, G)$ by conjugation. Then for any  $\rho\in\hm(\Gamma, G)$ and  $g\in G$, the map $\phi_g:\stab_G(\rho)\to \stab_G(g\rho g^{-1})$ defined by $h\mapsto ghg^{-1}$ is a bijection.
\end{prop}

\begin{remark}\label{trivialstrata}
Obviously, since $p>2$, $|\strata_Z|= 2^r$, and $\SL_2(\F_q)$ acts trivially and thus uniformly on $\strata_Z$ (and when $p=2$, $|\strata_Z|= 1$).  
\end{remark}

\begin{lemma}\label{sdstablemma}
For all $\rho\in \strata_D$, $|\stab_G(\rho)|=q-1$.  In other words, $G$ acts on $\strata_D$ uniformly of order $q-1$.
\end{lemma}
\begin{proof}
Let $\rho\in\strata_D$. Then there exists $g\in G$ so that $g\rho g^{-1}\in D^{\times r}$. By Proposition \ref{stabprop}, we count $|\stab_G(g\rho g^{-1})|$.  Suppose that $g\rho g^{-1}=\left(\mm{a_1}{0}{0}{a_1^{-1}},\ldots,\mm{a_r}{0}{0}{a_r^{-1}}\right).$  Since diagonal matrices commute, $D\subset \stab_G(g\rho g^{-1})$.  Next, since $\rho\notin \strata_Z$, there exists $i$ so $a_i\not=\pm 1$. Suppose that $B=\mm{a}{b}{c}{d}\in \stab_G(g\rho g^{-1})$, then $\mm{a_i}{0}{0}{a_i^{-1}}=B\mm{a_i}{0}{0}{a_i^{-1}}B^{-1}$; implying $B$ is diagonal.  Thus, $\stab_G(g\rho g^{-1})\subset D$, and so $|\stab_G(\rho)|=|\stab_G(g\rho g^{-1})|=|D|= q-1$.
\end{proof}

\begin{lemma}\label{sustablemma}    
If $\rho\in\strata_U$, then $|\stab_G(\rho)|=2q$; that is, $G$ acts uniformly of order $2q$ on $\strata_U$.
\end{lemma}

\begin{proof}
Let $\rho=(A_1,...,A_r)\in\strata_U$. Then by definition, there exists $g\in G$ so for all $1\leq i\leq r$, $gA_ig^{-1}=\mm{\pm1}{a_i}{0}{\pm1}$ for $a_i\in \F_q$, and there exists at least one $j$ so $a_j\not=0$. Let $B=\mm{\pm1}{k}{0}{\pm1}\in U$. Then for all $i$, $$B(gA_ig^{-1})B^{-1} =\mm{\pm1}{k}{0}{\pm1}\mm{\pm1}{a_i}{0}{\pm1}\mm{\pm1}{-k}{0}{\pm1}=\mm{\pm1}{a_i}{0}{\pm1}.$$
Thus, $U\subset\stab_G(g\rho g^{-1})$. Conversely, let $B=\mm{a}{b}{c}{d}\in\stab_G(g\rho g^{-1})$. Then, $B(gA_jg^{-1})B^{-1}=gA_jg^{-1}$, and so
$$\mm{\pm1 - a_jac}{a_ja^2}{-a_jc^2}{\pm1+a_jac}=\mm{\pm1}{a_j}{0}{\pm1}.$$
Thus, since $a_j\neq 0$, we obtain that $c=0$ and $a=\pm1$. This forces $d=\pm1=a$ since $\det(B)=1$.  Thus, $B\in U$ and so $\stab_G(g\rho g^{-1})\subset U$. We have shown $\stab_G(g\rho g^{-1})=U$.  Since $|U|=2q$, Proposition \ref{stabprop} implies the result.
\end{proof}

\begin{lemma}\label{sdbarstablemma}
If $\rho\in\strata_{\overline{D}}$, then $|\stab_G(\rho)|=q+1$; that is, $G$ acts uniformly of order $q+1$ on $\strata_{\overline{D}}$.
\end{lemma}

\begin{proof}
We first prove that there are exactly $(q-1)/2$ values of $t$ in $\F_q$ so that $x^2-tx+1$ is irreducible in $\F_q[x]$.  Any solution to $x^2-tx+1=0$ in $\overline{\F}_q$, by the quadratic formula, has the form $x=\frac{t\pm\sqrt{t^2-4}}{2}.$  So $x^2-tx+1$ is irreducible if and only if $t^2-4$ is a quadratic non-residue; that is, $y^2=t^2-4$ does not have a solution in $\F_q$.  This equation is equivalent to $t^2-y^2 = 4$ and thus $\left(\frac{t}{2}\right)^2-\left(\frac{y}{2}\right)^2= 1$ since $p$ is odd.

But the variety $z^2-w^2=1$ is isomorphic to the variety $uv=1$ via $u=z+w$ and $v=z-w$.  However, $uv=1$ is isomorphic to $\GL_1(\F_q)$ and thus has $q-1$ solutions.

So there exists $q-1$ pairs $(t/2,y/2)$ of solutions. This implies that for all $t/2\neq \pm 1$ there exists two values of $y$ (namely $\pm y$), and if $t/2=\pm 1$ then $y=0$. Consequently, there are $\frac{(q-1)-2}{2}+2=\frac{q-3+4}{2}=\frac{q+1}{2}$ choices for $t/2$. This implies that there are $(q+1)/2$ choices for $t$ that yield a reducible polynomial $x^2-tx+1$, and $q-(q+1)/2=(q-1)/2$ choices that do not.

Let $\rho\in \strata_{\overline{D}}$.  By definition, there exists $g\in \SL_2(\overline{\F}_q)$ such that $g\rho g^{-1}\in \overline{D}^{\times r}$ and by the No-Mixing Theorem (Proposition \ref{nomixing}) we know that $g\rho g^{-1}\notin D^{\times r}$.  Let $g\rho g^{-1}=(A_1,...,A_r)$.  Thus each $A_i$ is diagonal with eigenvalues either $\pm 1$ or in $\F_q[x]/(x^2-tx+1)$ where $x^2-tx+1$ is irreducible; but $g\rho g^{-1}$ is not central.  From the proof of Lemma \ref{sdstablemma} the only elements that stabilize such a representation are $\overline{D}$.  However, we must determine the number of elements in $\SL_2(\F_q)$ that stabilize $\rho$.  Observe that $B\in\SL_2(\F_q)$ stabilizes $\rho$ if and only if $gBg^{-1}$ stabilizes $g\rho g^{-1}$.  Thus we must count the number of elements in $\overline{D}$ that are conjugate to elements in $\SL_2(\F_q)$ via $g$.  Since there are $(q-1)/2$ irreducible polynomials and each gives exactly two distinct eigenvalues, there are $2(q-1)/2=q-1$ such diagonal matrices 
whose eigenvalues are not in $\F_q$ yet are conjugate to an element in $\SL_2(\F_q)$.  Note that by the No-Mixing Theorem (item (2)) that $g$ does conjugate each of these diagonal matrices to $\SL_2(\F_q)$, and furthermore (by item (1)) if $gBg^{-1}$ is not $\pm \id$ but in $D$, then $B$ cannot be in $\SL_2(\F_q)$.  Thus, we need only add in $\pm \id$ from $D$ to the $q-1$ diagonal elements that come from $\overline{D}-D$, and so there are $(q-1)+2=q+1$ elements in $\stab_G(\rho)$. 
\end{proof}

\begin{lemma}\label{snrstablemma}
If $\rho\in\strata_{\NR}$, then $|\stab_G(\rho)|=2$.
\end{lemma}

\begin{proof}
Suppose that $\rho\in\strata_{\NR}$, that is, reducible and non-Abelian.  By Proposition \ref{stabprop}, we can assume that $\rho=(A_1,...,A_r)$ has already been put into upper-triangular form. By the No-Mixing Theorem the eigenvalues of any $A_i$ are in $\F_q$.  If all such eigenvalues are $\pm 1$, then $\rho$ is Abelian and hence not in $\strata_{\NR}$.  Thus there is some $A_{i_1}=\left(\begin{array}{cc}a&x\\0&a^{-1}\end{array}\right)$ with $a\not=\pm 1$.  Since $a\not=\pm 1$, conjugating by $\left(\begin{array}{cc}1& \frac{x}{a-a^{-1}}\\0&1\end{array}\right)$ we can assume that $x=0$.  At this point, since $\rho\not\in \strata_D$ there is $A_{i_2}=\left(\begin{array}{cc}b&y\\0&b^{-1}\end{array}\right)$ with $y\not=0$.  We have already seen that the elements that stabilize $A_{i_1}$ are diagonal.  So $\{\pm \id\}\subset \stab_G(\rho)\subset \overline{D}$.  Now take $C=\mm{c}{0}{0}{c^{-1}}\in\stab_{G}(\rho)$.  Therefore, $A_{i_2}=CA_{i_2}C^{-1}.$  This implies that $$\mm{b}{y}{0}{b^{-1}}=\mm{c}{0}{0}{c^{-1}
}\mm{b}{y}{0}{b^{-1}}\mm{c^{-1}}{0}{0}{c}=\mm{b}{c^2y}{0}{b^{-1}},$$ which implies that $c^2y=y$, or $c=\pm 1$ since $y\not=0$. Thus, $C\in Z$ and so $\stab_G(\rho)=Z$, as required.
\end{proof}

\begin{remark}
Lemma \ref{snrstablemma} finishes the argument in Proposition \ref{strataprop} which shows that $\SL_2(\F_q)$ acts uniformly of order 2 on $\strata_{N}$, which together with Remark \ref{trivialstrata} and Lemmata \ref{sdstablemma}, \ref{sustablemma}, \ref{sdbarstablemma} finish the proof of part (b) of Proposition \ref{strataprop} which says that $\SL_2(\F_q)$ acts uniformly on each of the subsets defined in Definition \ref{stratadef}.
\end{remark}

\section{Counting Strata and Orbits}
Recall our convention that $G$ denotes $\SL_2(\F_q)$ with $q=p^k$, $p$ an odd prime.  In this section we will count the number of points as a polynomial in $q$ (for every $r$) in each stratum defined in Definition \ref{stratadef}, and thereby likewise determine the number of orbits in $\QS_r(G)$ (proving Theorem \ref{theorema}), by Propositions \ref{uniformactionthm} and \ref{strataprop}.  

By Remark \ref{trivialstrata} and the fact that $G$ acts trivially on $\strata_Z$, $|\strata_Z|=|\strata_Z/G|=2^r$.  We next address the diagonalizable stratum.

Let $W=\left\{\mm{1}{0}{0}{1},\mm{-1}{0}{0}{-1},\mm{0}{-1}{1}{0},\mm{0}{1}{-1}{0}\right\}$ be the Weyl group in $\SL_2(\F_q)$, and note that it acts on $D^{\times r}$ by simultaneously permuting the diagonal entries.

\begin{prop}\label{stratasdcount}
$|\strata_D/G|=\frac{(q-1)^r-2^r}{2}$ and $|\strata_D|=\frac{(q-1)^r-2^r}{2}q(q+1)$
\end{prop}

\begin{proof}
Consider the following commutative diagram, where $\varphi=\pi\circ \iota$ by definition: $$\xymatrix{(D^{\times r}-\strata_Z)/ W \ar@{^{(}->}[rr]^{\iota} \ar[drr]_\varphi & & \strata_D/W \ar@{->>}[d]^{\pi}\\ & & \strata_D/ G}$$  We first prove that $\varphi$ is bijective.  Suppose $\varphi([\rho_1]_W)=[\rho_1]_G=\left[\rho_2\right]_G=\varphi([\rho_2]_W)$, so there exists $x\in G$ such that $\rho_1=x\rho_2x^{-1}$. Since $\rho_1=(A_1,...,A_r),\rho_2=(B_1,...,B_r)\in D^{\times r}-\strata_Z$, there is an index $1\leq j\leq r$ so that $A_j \notin Z$. 
Then it is easy to see that $x$ preserves or swaps the two eigenvectors of $A_j$ and $B_j$, and so $x=dw$ where $d\in D$ and $w\in W$.  Then, $\rho_1=(dw)\rho_2(dw)^{-1}$ which implies $\rho_1=d^{-1}\rho_1 d=w\rho_2w^{-1}$ since $d\in D$ and $\rho_1\in D^{\times r}-\strata_Z$.  Thus, $\left[\rho_1\right]_W=\left[\rho_2\right]_W$, and $\varphi$ is injective.  To show that $\varphi$ is surjective, let $\left[\rho\right]_G\in \strata_D\frk G$. Then, there exists a $g\in G$ so that $g\rho g^{-1}\in D^{\times r}-\strata_Z$. If we consider $\left[g\rho g^{-1}\right]_W\in(D^{\times r}-\strata_Z)\frk W$, we have $\varphi\left(\left[g\rho g^{-1}\right]_W\right)=\pi\left(\iota\left(\left[g\rho g^{-1}\right]_W\right)\right)=\pi\left(\left[g\rho g^{-1}\right]_W\right)=\left[g\rho g^{-1}\right]_G=\left[\rho\right]_G$, showing $\varphi$ is surjective.  Therefore, $|\strata_D/G|=|(D^{\times r}-\strata_Z)\frk W|$.

Next, we show that $W$ acts uniformly of order $2$ on $D^{\times r}-\strata_Z$.  Since $\pm \id$ fixes all representations, the cardinality of the $W$-stabilizer of any representation in $D^{\times r}-\strata_Z$ must be greater than or equal to $2$. On the other hand, for any $\rho=(A_1,...,A_r) \in D^{\times r}-\strata_Z$, there exists any index $i$ so $A_i=\mm{a_i}{0}{0}{a_i^{-1}}\notin Z$, which implies that $a_i\neq a_i^{-1}$. The two non-central elements in $W$ do not stabilize this matrix since they permute the non-equal diagonal entries. Thus, the cardinality of the $W$-stabilizer of any element in $D^{\times r}-\strata_Z$ must be less than or equal to $2$.  Therefore, we conclude that $W$ acts uniform of order $2$.  

Any tuple $\rho\in D^{\times r}-\strata_Z$ is of the form $\rho=\left(\mm{a_1}{0}{0}{a_1^{-1}},...,\mm{a_r}{0}{0}{a_r^{-1}}\right),$
and so we have $q-1$ choices for each $a_i$ for $1\leq i \leq r$. This results in a total of $(q-1)^r$ choices for such tuples, but we must remove $\strata_Z$, which has cardinality $2^r$ by Remark \ref{trivialstrata}.  Then, by Proposition \ref{uniformactionthm}, $|(D^{\times r}-\strata_Z)\frk W|=2\frac{(q-1)^r-2^r}{|W|}=\frac{(q-1)^r-2^r}{2},$ and since $\phi$ is bijective $|\strata_D/G|=\frac{(q-1)^r-2^r}{2}$ as well.  

Lastly, Proposition \ref{uniformactionthm} and Lemma \ref{sdstablemma} give $$|\strata_D|=\frac{|G|}{q-1}|\strata_D \frk G|=\frac{q(q-1)(q+1)}{q-1}\left(\frac{(q-1)^r-2^r}{2}\right)=q(q+1)\left(\frac{(q-1)^r-2^r}{2}\right).$$
\end{proof}

\begin{prop}\label{stratasdbarcount}
$|\strata_\Dbar/G|=\frac{(q+1)^r-2^r}{2}$ and $|\strata_\Dbar|=\frac{(q+1)^r-2^r}{2}q(q-1)$
\end{prop}

\begin{proof}
Lemma \ref{sdbarstablemma} shows that $G$ acts uniformly of order $q+1$ on $\strata_\Dbar$, thus by Proposition \ref{uniformactionthm}:  $|\strata_\Dbar|=\frac{|G|}{q+1}|\strata_\Dbar \frk G|=\frac{q(q-1)(q+1)}{q+1}|\strata_\Dbar \frk G|=q(q-1)|\strata_\Dbar \frk G|.$  So it suffices to prove that $|\strata_\Dbar/G|=\frac{(q+1)^r-2^r}{2}$.

Let $\mathcal{E}$ be the set of diagonal matrices that are either central, or whose eigenvalues $\lambda$ are zeros of an irreducible polynomial $x^2-tx+1\in \F_q[x]$.  From Lemma \ref{sdbarstablemma}, we know there are exactly $(q-1)/2$ values of $t$ in $\F_q$ so that $x^2-tx+1$ is irreducible in $\F_q[x]$.  Thus, $|\mathcal{E}|=2(q-1)/2+|Z|=(q-1)+2=q+1$ since for each such irreducible polynomial we get exactly two distinct diagonal matrices $\mm{\lambda}{0}{0}{\lambda^{-1}}$ and $\mm{\lambda^{-1}}{0}{0}{\lambda}$.

For the same reason as in the proof of Proposition \ref{stratasdcount}, $W$ acts uniformly of order 2 on $\mathcal{E}^{\times r}-\strata_Z$, thus by Proposition \ref{uniformactionthm}, $|(\mathcal{E}^{\times r}-\strata_Z)\frk W|=2\frac{(q+1)^r-2^r}{4}=\frac{(q+1)^r-2^r}{2}$.  So it suffices to prove that $\strata_{\Dbar}$ is in one-to-one correspondence with $(\mathcal{E}^{\times r}-\strata_Z)\frk W$.  

For any representation $\rho\in \strata_\Dbar$, by definition, there exists $g\in \SL_2(\overline{\F}_q)$ such that $g\rho g^{-1}$ is diagonal.  By the No-Mixing Theorem (Proposition \ref{nomixing}), $g\rho g^{-1}\in \mathcal{E}^{\times r}-\strata_Z$.  We claim that this association defines a bijection $\varphi:\strata_\Dbar/G\to(\mathcal{E}^{\times r}-\strata_Z)\frk W$.

We first show $\varphi$ is well-defined.  Let $g\rho g^{-1}=\left(\mm{e_1}{0}{0}{e_1^{-1}},...,\mm{e_r}{0}{0}{e_r^{-1}}\right)$ and $h\rho h^{-1}=\left(\mm{f_1}{0}{0}{f_1^{-1}},..., \mm{f_r}{0}{0}{f_r^{-1}}\right)$ where $h,g\in \SL_2(\overline{\F}_q)$.  Then $gh^{-1}$ conjugates $h\rho h^{-1}$ into $g\rho g^{-1}$.  As shown in the proof of Proposition \ref{stratasdcount}, using Equations (1)-(4), $gh^{-1}=dw$ for diagonal $d$ and $w\in W$.  Thus, since $d$ acts trivially on diagonal matrices, we conclude that $g\rho g^{-1}$ and $h\rho h^{-1}$ are conjugate via an element of $W$.  This shows $\varphi$ is well-defined.

Similarly, to show that $\varphi$ is injective, suppose that $h \rho_1h^{-1}=g\rho_2g^{-1}$ in $\mathcal{E}^{\times r}-\strata_D$ (up to the action of $W$).  Again, $gh^{-1}$ conjugates non-central diagonal $h\rho _1 h^{-1}$ into non-central diagonal $g\rho_2g^{-1}$, and thus, $gh^{-1}=dw$ where $d$ is diagonal and $w\in W$.  We then have $dwh\rho_2 h^{-1}w^{-1}d^{-1}=h\rho_1 h^{-1}$.  Since diagonal matrices act trivially on diagonal representations, we conclude $w h\rho_2 h^{-1} w^{-1}=h\rho_1 h^{-1}$.  Similar to the proof of Lemma \ref{basefieldreductionlemma}, since $\rho_1$ cannot be upper or lower triangular to begin with (no simultaneous eigenvalues over $\F_q$), we know that $h^{-1}=h_0^{-1}\delta$ where $h_0^{-1}$ has the form $\mm{b}{b}{\lambda-a}{\lambda^{-1}-a}$ where $a,b\in \F_q$ and $\lambda$ is a zero of an irreducible polynomial $x^2-tx+1\in \F_q[x]$, and $\delta$ is diagonal in $\SL_2(\overline{\F}_q)$.  Therefore we have $w\delta ^{-1}h_0\rho_2h_0^{-1}\delta w^{-1}=\delta^{-1}h_0\rho_1h_
0^{-1}\delta$.  Note that that $h_0$ has columns that form a pair of linearly independent simultaneous eigenvectors for each component of $\rho_1$, and thus $h_0\rho_1 h_0^{-1}$ is diagonal, and so $\delta$ acts trivially on it.  Since $w$ preserves the diagonal form of $\rho_1$, we conclude that $w\delta^{-1}h_0\rho_2h_0^{-1}\delta w^{-1}=wh_0\rho_2h_0^{-1}w^{-1}$.  Thus we have $h_0^{-1}wh_0\rho_2h_0^{-1}w^{-1}h_0=\rho_1$.  We claim that $h_0^{-1}wh_0$ is in $\SL_2(\F_q)$; that is, defined over the base field.  If $w$ is central this is obvious.  Otherwise, \begin{align*}h_0^{-1}wh_0&=
\left(
\begin{array}{cc}
 b \left(2 a-\lambda
   -\frac{1}{\lambda }\right)
   & 2 b^2 \\
 -2 a^2+\frac{2 a
   \left(\lambda
   ^2+1\right)}{\lambda
   }-\frac{\lambda
   ^4+1}{\lambda ^2} & b
   \left(-2 a+\lambda
   +\frac{1}{\lambda }\right)
   \\
\end{array}
\right)\\
&=
\left(
\begin{array}{cc}
 b \left(2 a-t\right)
   & 2 b^2 \\
 -2 a^2+2 at -(t^2-2) & b
   \left(-2 a+t\right)
   \\
\end{array}
\right).\end{align*}

Since the trace $t:=\lambda+\lambda^{-1}$ is in $\F_q$, it is apparent that the claim holds.  Therefore, $[\rho_1]_G=[\rho_2]_G$ and $\varphi$ is one-to-one.

Lastly, the converse of part (2) of the No-Mixing Theorem (Proposition \ref{nomixing}), directly says that $\varphi$ is surjective.

\end{proof}

\begin{prop}\label{strataucount}
$|\strata_U/G|=\frac{(2q)^r-2^r}{\frac{q-1}{2}}$ and $|\strata_U|=(q+1)\left((2q)^r-2^r\right)$
\end{prop}

\begin{proof}
Let $T=\left\{\mm{x}{y}{0}{x^{-1}} \in G  \ | \ x\in\F_q^*,\ y \in\F_q\right\}$, which has order $q(q-1)$. Since upper-triangular matrices preserve upper-triangular matrices by conjugation, $T$ acts on $U^{\times r}$.  We claim the map $\varphi$ defined in the following diagram is a bijection:
$$\xymatrix{(U^{\times r}-\strata_Z)/ T \ar@{^{(}->}[rr]^{\iota} \ar[drr]_\varphi & & \strata_U/T \ar@{->>}[d]^{\pi}\\ & & \strata_U/ G}$$

To show that $\varphi$ is surjective, let $[\rho]_G\in \strata_U/G$.  Then by definition, there exists $g\in G$ so $g\rho g^{-1}\in U^{\times r}-\strata_Z$.  Then $\varphi([g\rho g^{-1}]_T)=\pi \circ\iota \left(\left[g\rho g^{-1}\right]_T\right)=\pi\left(\left[g\rho g^{-1}\right]_T\right)=\left[g\rho g^{-1}\right]_G=\left[\rho\right]_G$.

To prove that $\phi$ is injective, let $\rho_1,\rho_2\in U^{\times r}-\strata_Z$ and suppose we have $[\rho_1]_G=\varphi\left(\left[\rho_1\right]_T\right)=\varphi\left(\left[\rho_2\right]_T\right)=[\rho_2]_G$. Then, there exists a $g\in G$ so that $g\rho_1 g^{-1}=\rho_2$. Let $g=\mm{w}{x}{y}{z}$. Since $\rho_1\in U^{\times r}-\strata_Z$, there exists a component of $\rho_1$, call it $A_i$, so that $A_i\in U-Z$.  In other words, $A_i=\mm{\pm1}{a_i}{0}{\pm1}$ for some $a_i\in\F_q^*$.  Clearly, $B_i:=gA_ig^{-1}\not=\pm \id$.  Thus $B_i=\mm{\pm1}{b_i}{0}{\pm1}$ where $b_i\in \F_q^*$.  Note that the $\pm 1$'s correspond since the eigenvalues are repeated and conjugation does not change their value.  Then since $g$ preserves the eigenvector $(1,0)^\dagger$, one concludes that $g\in T$, and so $\left[\rho_1\right]_T=\left[\rho_2\right]_T$; showing $\varphi$ is injective.

Next we show that $T$ acts uniformly on $U^{\times r}-\strata_Z$.  Since $T\subset G$, we immediately have $\stab_T(\rho)\subset \stab_G(\rho)$. Conversely, we showed in Lemma \ref{sustablemma} that $\stab_G(\rho)=U$. But $U\subset T$, and so we obtain $\stab_G(\rho)\subset\stab_T(\rho)$.  So for any $\rho\in U^{\times r}-\strata_Z$, $U=\stab_G(\rho)=\stab_T(\rho)$, and therefore $T$ acts uniformly of order $2q$ on $U^{\times r}-\strata_Z$ since $|U|=2q$.  Therefore, by Proposition \ref{uniformactionthm} and the bijection $\varphi$, $|\strata_U/G|=\left|(U^{\times r}-\strata_Z)\frk T\right|=2q\frac{(2q)^r-2^r}{q(q-1)}=\frac{(2q)^r-2^r}{\frac{q-1}{2}}.$  Then by Proposition \ref{uniformactionthm} and Lemma \ref{sustablemma}, we conclude $\left|\strata_U \frk G\right|=2q \frac{|\strata_U|}{|G|}$ and so $|\strata_U|= \frac{q(q+1)(q-1)}{2q}\frac{(2q)^r-2^r}{\frac{q-1}{2}}= (q+1)((2q)^r-2^r).$

\end{proof}

Let $(T^{\times r})^*=T^{\times r}-(\strata_D\cup \strata_U\cup \strata_Z)$.  By Lemma \ref{snrstablemma}, $T$ acts on $(T^{\times r})^*$ uniformly of order 2.

\begin{lemma}
$\sn/G$ is in bijective correspondence with $(T^{\times r})^*/T.$
\end{lemma}

\begin{proof}
By Proposition \ref{strataprop}, $(T^{\times r})^*\subset \sn$, and so there is a mapping $\varphi:(T^{\times r})^*/T\to \sn/G$.  By Lemma \ref{basefieldreductionlemma} and Proposition \ref{strataprop}, every $[\rho]_G\in \sn/G$ is represented by an element in $(T^{\times r})^*$, and so $\varphi$ is onto.  Now let $[\rho_1]_G=[\rho_2]_G$ where $\rho_1$ and $\rho_2$ are in $(T^{\times r})^*$.  Then there exists $g\in G$ so $g\rho_1 g^{-1}=\rho_2$.  Let $g=\mm{w}{x}{y}{z}$.  By assumption $\rho_1$ and $\rho_2$ have corresponding non-central upper-triangular components $gA_i g^{-1}=B_i$.  If any non-central component $A_i\in U$, then $A_i=\mm{\pm1}{a_i}{0}{\pm1}$ for some $a_i\in\F_q^*$.  Since $a_i\not=0$, $B_i:=gA_ig^{-1}\not=\pm \id$.  Thus $B_i=\mm{\pm1}{b_i}{0}{\pm1}$ where $b_i\in \F_q^*$, and the $\pm 1$'s correspond since the eigenvalues are repeated and conjugation does not change their value.  So, as in the proof of Lemma \ref{sustablemma}, we have that $g\in T$.  

Otherwise, all non-central components $A_i$ have distinct eigenvalues (since $\rho_1, \rho_2$ are in $(T^{\times r})^*$, in this case, there must be at least two such components).  Consider any such component $A_k=\mm{a_k}{b_k}{0}{a_k^{-1}}$ and its corresponding component $B_k=\mm{e_k}{f_k}{0}{e_k^{-1}}$.  Then for each such component,
\begin{align*}
\mm{e_k}{f_k}{0}{e_k^{-1}}=&\mm{w}{x}{y}{z}\mm{a_k}{b_k}{0}{a_k^{-1}}\mm{z}{-x}{-y}{w}\\
  =&
  \left(
\begin{array}{cc}
 w z a_k-y
   \left(\frac{x}{a_k}+w
   b_k\right) & w \left(x
   \left(-a_k\right)+\frac{x}{
   a_k}+w b_k\right) \\
 y z a_k-y
   \left(\frac{z}{a_k}+y
   b_k\right) & \frac{w
   z}{a_k}-x y a_k+w y b_k \\
\end{array}
\right).
 \end{align*}
This implies that $y z a_k-y\left(\frac{z}{a_k}+yb_k\right)=0$, or $y(z(a_k-a_k^{-1})-b_ky)=0$. Now, either $y=0$ or $y\neq 0$. If $y=0$, then $g\in T$. Otherwise, since the eigenvalues of $A_k$ are distinct, $z=\frac{b_ky}{a_k-a_k^{-1}}$.  Therefore, for all indices $k,j$ corresponding to non-central components,  
$\frac{b_ky}{a_k-a_k^{-1}}=\frac{b_jy}{a_j-a_j^{-1}}$ and thus $\frac{b_k}{a_k-a_k^{-1}}=\frac{b_j}{a_j-a_j^{-1}}$ since $y\neq 0$.
Thus, conjugating $\rho_1$ by $\left(
\begin{array}{cc}
 1 &
   \frac{b_k}{a_k-a_k^{-1}
   } \\
 0 & 1 \\
\end{array}
\right)$ diagonalizes each component simultaneously.  However, $(T^{\times r})^*\cap \strata_D=\emptyset$.  Thus, it must be the case that $y=0$.  Said differently, there exists $g\in T$ so $g\rho_1 g^{-1}=\rho_2$, implying $[\rho_1]_T=[\rho_2]_T$ and showing $\varphi$ is injective.

\end{proof}

Let $T^*=T-U\cup D$.  For any two elements $A_1=\mm{a_1}{b_1}{0}{a_1^{-1}}$ and $A_2=\mm{a_2}{b_2}{0}{a_2^{-1}}$ in $T^*$, we will say they are {\it diagonally compatible} if $\frac{b_1}{a_1-a_1^{-1}}=\frac{b_2}{a_2-a_2^{-1}}$.  In this case, we will write $A_1\sim_D A_2$.  Define $$\strata_{T\!D}=\left\{(A_1,...,A_r)\in (T^*\cup Z)^{\times r}\ | \text{ if }A_i,A_j\in T^*\text{ then } A_i\sim_D A_j\right\}-\strata_Z.$$

\begin{lemma}
$\strata_D\cap T^{\times r}=\strata_{T\!D}\cup (D^{\times r}-Z^{\times r})$. 
 \end{lemma}
 
\begin{proof}
By definition, $D^{\times r}-Z^{\times r}\subset \strata_D\cap T^{\times r}$ and $\strata_{T\!D}\subset T^{\times r}$.  So to prove the inclusion $\strata_{T\!D}\cup (D^{\times r}-Z^{\times r})\subset\strata_D\cap T^{\times r}$, we need to show that $\strata_{T\!D}\subset \strata_D$. Let $\rho=(A_1,...,A_r)\in \strata_{T\!D}$. By definition $\rho$ is non-central, and all $A_i\in T^*$ are diagonally compatible. Since $\mm{1}{\frac{b_i}{a_i-a_i^{-1}}}{0}{1}$ diagonalizes $A_i$, it then diagonalizes $\rho$.  Therefore, $\rho \in \strata_D$, as required.

Conversely, let $\rho=(A_1,...,A_r)\in \strata_D\cap T^{\times r}$.  Then for all $1\leq i\leq r$,
$A_i=\mm{a_i}{b_i}{0}{a_i^{-1}}$. If $b_i=0$ for all $1\leq i\leq r$, then $\rho\in D^{\times r}\subset\strata_{T\!D}\cup D^{\times r}$.
  
Otherwise, some $b_i\not=0$.  Without loss of generality, suppose that $b_1\neq 0$. Then for any $k$, we want to show:
\begin{enumerate}
\item if $b_k=0$, then $a_k=\pm1$;
\item if $b_k\neq 0$, then $a_k\neq a_k^{-1}$ and $\frac{b_1}{a_1-a_1^{-1}}=\frac{b_k}{a_k-a_k^{-1}}$.
\end{enumerate}

First note that if any component $A_k\in U-Z$, $\rho$ would not be diagonalizable.  Thus $A_1$ satisfies condition (2).  Also, this implies that any component $A_k$ where $b_k\not=0$ automatically satisfies $a_k\not=a_k^{-1}$, and thus is in $T^*$.

Since $A_1,...,A_r$ share $(1, 0)^\dagger$ as a common eigenvector and $\rho$ can be diagonalized via a determinant 1 matrix $g$, we know that $g$ may be taken to have the form $\mm{1}{z}{0}{1}$.  Solving for $z$ in $gA_1g^{-1}=\mm{a_1}{0}{0}{a_1^{-1}}$ we conclude that $z=\frac{b_1}{a_1-a_1^{-1}}$.

Conjugating any other $A_k$ by $g$, we conclude that if $b_k=0$, then it must be the case that $a_k=\pm1$, which establishes item (1).  On the other hand, if $b_k\neq0$, then $\frac{b_1}{a_1-a_1^{-1}}=\frac{b_k}{a_k-a_k^{-1}}$, as required to establish item (2).  Since (1) and (2) are satisfied, $\rho\in \strata_{T\!D}\subset \strata_{T\!D}\cup D^{\times r}$.
\end{proof}

\begin{lemma}
${\displaystyle |\strata_D\cap T^{\times r}|=q\left((q-1)^r-2^r\right)}$

\end{lemma}
 
\begin{proof}
By the previous lemma, we need to count $\strata_{T\!D}\cup (D^{\times r}-Z^{\times r})$.  Since by definition, $\strata_{T\!D}$ contains no diagonal matrices, this union is disjoint.  Clearly, $|D^{\times r}-Z^{\times r}|=(q-1)^r-2^r$.  Now let $\rho=(A_1,...,A_r)\in \strata_{T\!D}$ where $A_i=\mm{a_i}{b_i}{0}{a_i^{-1}}$.  The factors of $\rho$ are either in $Z$ or they are in $T^*$, and $\rho\not\in\strata_Z$.  Since there are $r$ factors, we enumerate over the number of factors that are in $Z$ (at most $r-1$).  If no factor is in $Z$, then we have $(q-3)^r$ choices for the diagonal elements ($a_i\not=0,\pm 1$) and since the upper-right entries are all determined by the value of only one (and non-zero) there is a further $q-1$ choices for that coordinate.  To see this notice that by condition (2) in the previous lemma, $b_i=b_j\frac{a_i-a_i^{-1}}{a_j-a_j^{-1}}$ for any $i$ or $j$.  Thus, in that case there are $(q-3)^r(q-1)$ choices for $\rho$.  Supposing now there are $k$ factors in $Z$ (where $1\leq k\
leq r-1$) in 
fixed position, we then have $(q-3)^{r-k}(q-1)$ choices for those $k$ factors (for exactly the same reason as in the $k=0$ case) times the number of choices for central components; namely $2^k$.  However, there are exactly $\binom{r}{k}$ choices for the positions of those $k$ central components.  So we further must multiply by $\binom{r}{k}$.  Enumerating over $k$, we conclude that there are $\sum_{k=0}^{r-1}\binom{r}{k}2^k(q-3)^{r-k}(q-1)=((q-1)^r-2^r)(q-1)$ representations in $\strata_{T\!D}$.  The result follows.
\end{proof}

\begin{prop}\label{stratanrcount}
 $$|\sn/G|=\frac{2}{q(q-1)}\left((q-1)^rq^r-(2q)^r-q((q-1)^r-2^r)\right)$$
and 
$$|\sn|=(q+1)\left((q-1)^rq^r-(2q)^r-q((q-1)^r-2^r)\right)$$
\end{prop}

\begin{proof}
First note that $T$ acts uniformly of order 2 on $(T^{\times r})^*$, and by definition $(T^{\times r})^*=T^{\times r}-\strata_U\cup\strata_D\cup\strata_Z$.  Since $\strata_U\cap T^{\times r}=U^{\times r}-Z^{\times r}$, $\strata_Z\cap T^{\times r}=Z^{\times r}$, and $\strata_D\cap T^{\times r}=\strata_{T\!D}\cup (D^{\times r}-Z^{\times r})$, we conclude that $|(T^{\times r})^*|=(q-1)^rq^r-((2q)^r-2^r)-2^r-q((q-1)^r-2^r),$ and consequently, $|\sn/G|=|(T^{\times r})^*/T|=\frac{2}{q(q-1)}\left((q-1)^rq^r-(2q)^r-q((q-1)^r-2^r)\right).$

Therefore, since $G$ acts uniformly of order 2 on $\sn$, we conclude that: \begin{align*}&|\sn|=\frac{q(q-1)(q+1)}{2}|\sn/G|\\&=\frac{q(q-1)(q+1)}{2}\frac{2}{q(q-1)}\left((q-1)^rq^r-(2q)^r-q((q-1)^r-2^r)\right)\\ &=(q+1)\left((q-1)^rq^r-(2q)^r-q((q-1)^r-2^r)\right)\end{align*}
\end{proof}

\section{Galois Action and Absolutely Irreducibles}

The $\F_q$-points of the GIT quotient $\X_r(\C)$ correspond to the Zariski closed $\SL_2(\overline{\F}_q)$-conjugation orbits in $\mathrm{Hom}(\fr_r,\SL_2(\overline{\F}_q))$.  The points in $\strata_U$ and those in $\sn$ are upper-triangular and therefore not conjugate to elements in $\strata_Z$ and $\strata_D$ (respectively).  However, the coordinate ring of $\X_r(\C)$ is generated by traces, and the representations in $\strata_U$ and those in $\sn$ cannot be distinguished from those in $\strata_Z$ and $\strata_D$ (respectively) via traces alone.  Therefore, the elements in $\strata_U$ and those in $\sn$ do not have closed orbits.  Clearly, all other representations in $\mathrm{Hom}(\fr_r,\SL_2(\overline{\F}_q))$ do have closed orbits.  Let $\QS^*_r(\F_q)\subset \QS_r(\F_q)$ be the set of all closed orbits.  Then $\QS^*_r(\F_q)=(\strata_Z\cup\strata_D\cup \strata_{\overline{D}}\cup\strata_{\AI})/\SL_2(\F_q)$.  

Naturally, there is a mapping from $\QS^*_r(\F_q)$ onto the $\F_q$-points of $\X_r(\overline{\F}_q)$.  We now show that this mapping is injective on the Abelian locus, and 2-to-1 on the irreducible locus.

First note that for the mapping to be non-injective, we must have representations that are not conjugate via $\SL_2(\F_q)$ yet are conjugate via $\SL_2(\overline{\F}_q)$.  Moreover, any such equivalence must preserve the locus of $\F_q$-points in $\mathrm{Hom}(\fr_r,\SL_2(\overline{\F}_q))$.  We will think of such an action on a given stratum as {\it Galois} since it preserves the $\F_q$-points of that stratum yet identifies points using coordinates in an extension.  

Now let $g=
\left(
\begin{array}{cc}
 a & b \\
 c & d \\
\end{array}
\right)$ be in $\SL_2(\overline{\F}_q)$, and suppose that it preserves the $\F_q$-points in $\mathrm{Hom}(\fr_r,\SL_2(\overline{\F}_q))$; namely, $\mathrm{Hom}(\fr_r,\SL_2(\F_q))$.  Then, since all such homomorphisms are $r$-tuples of elements in $\SL_2(\F_q)$ and the action is simultaneous conjugation, $g$ must also preserve $A=\left(
\begin{array}{cc}
 1 & 1 \\
 0 & 1 \\
\end{array}
\right)$, and the transpose of $A$.  Indeed, $gAg^{-1}=\left(
\begin{array}{cc}
 1-a c & a^2 \\
 -c^2 & a c+1 \\
\end{array}
\right)$
and $gA^Tg^{-1}=\left(
\begin{array}{cc}
 b d+1 & -b^2 \\
 d^2 & 1-b d \\
\end{array}
\right)$.

Since the result must remain in $\SL_2(\F_q)$, we conclude that $a,b,c,d$ must be square-roots of elements in $\F_q$.  Moreover, since $ac$ and $bd$ must also be in the base field and $ad-bc=1$, this implies that all of $a,b,c,d$ are in $\F_q$ or all are in $\overline{\F}_q-\F_q$ (by Lemma \ref{residuelemma}).  Now for any two quadratic non-residues of $\F_q$, call them $x$ and $y$, we know that their ratio is a quadratic residue by Lemma \ref{residuelemma}.  Thus, $\sqrt{x}=\lambda \sqrt{y}$ for $\lambda=\frac{\sqrt{x}}{\sqrt{y}}\in \F_q$.  Thus, we conclude that if $g\in \SL_2(\overline{\F}_q)-\SL_2(\F_q)$, then $$g=\left(\begin{array}{cc}\lambda_{11}\sqrt{x}&\lambda_{12}/\sqrt{x}\\ \lambda_{21}\sqrt{x} & \lambda_{22}/\sqrt{x}\end{array}\right)=\left(\begin{array}{cc}\lambda_{11}&\lambda_{12}\\ \lambda_{21}&\lambda_{22}\end{array}\right)\left(\begin{array}{cc}\sqrt{x}&0\\0&1/\sqrt{x}\end{array}\right),$$ where
$\lambda_{ij}\in\F_q$ and $x$ is a quadratic non-residue of $\F_q$.

Consequently, the only additional equivalence that occurs over the algebraic closure, is conjugation by $\left(\begin{array}{cc}\sqrt{x}&0\\0&1/\sqrt{x}\end{array}\right)$ for a single non-residue $x$.  This action is a $\mathbb{Z}_2$ action on $\QS_r(\F_q)$.  Since it is diagonal, it acts trivially on the orbit spaces of diagonal and trivial strata, and for the same reason it acts uniform of order 2 on the quotient of the absolutely irreducible stratum.

Therefore, we have shown 
\begin{prop}\label{orbitstopoints}
The surjective mapping $\QS_r^*(\F_q)\to \X_r(\F_q)$ is injective over the reducible locus, and 2-to-1 over the irreducible locus.
\end{prop}

\section{Proof of Theorem \ref{epolyniomial}}\label{sectionproof}
Recall that Theorem \ref{epolyniomial} says that the $E$-polynomial for $\X_r(\SL_2(\C))$ is 
$$E_r(q)=(q-1)^{r-1}
   \left((q+1)^{r-1}-1\right)
   q^{r-1}+\frac{1}{2} q
   \left((q-1)^{r-1}+(q+1)^{r-1}
   \right)$$

\begin{proof}[Proof of Theorem \ref{epolyniomial}]
There is a bijection between the Zariski closed orbits (over the algebraic closure of $\F_q$) in $\hm(\fr_r,\SL_2(\F_q))$ and the $\F_q$-points in the GIT quotient $\X_r(\overline{\F}_q)$.  Let $\QS^*_r(\F_q)$ be the orbit space of orbits that consist of points whose $\SL_2(\overline{\F}_q)$-orbits are Zariski closed.  Thus, $\QS^*_r(\F_q)$ maps onto the $\F_q$-points of $\X_r(\overline{\F}_q)$.  With respect to Definition \ref{stratadef}, the points that have closed orbits are the trivial, diagonal, extendably diagonal and absolutely irreducible representations, since upper-triangular representations that are not diagonalizable do not have closed orbits. 
By Remark \ref{trivialstrata} and the fact that $G$ acts trivially on $\strata_Z$, $|\strata_Z|=|\strata_Z/G|=2^r$.  By Proposition \ref{stratasdcount}, $|\strata_D/G|=\frac{(q-1)^r-2^r}{2}$ and $|\strata_D|=\frac{(q-1)^r-2^r}{2}q(q+1)$.  By Proposition \ref{stratasdbarcount}, $|\strata_\Dbar/G|=\frac{(q+1)^r-2^r}{2}$ and $|\strata_\Dbar|=\frac{(q+1)^r-2^r}{2}q(q-1)$.  By Proposition \ref{strataucount}, $|\strata_U/G|=\frac{(2q)^r-2^r}{\frac{q-1}{2}}$ and $|\strata_U|=(q+1)\left((2q)^r-2^r\right)$.  And by Proposition \ref{stratanrcount}, $|\sn/G|=\frac{2}{q(q-1)}(q^r-q)((q-1)^r-2^r)$ and $|\sn|=(q+1)(q^r-q)((q-1)^r-2^r).$  Therefore, using Corollary \ref{sec:corollary13}, we obtain an explicit formula $$|\strata_{\AI}|=|\hm(\fr_r,\SL_2(\F_q))|-|\strata_Z|-|\strata_D|-|\strata_\Dbar|-|\strata_U|-|\sn|,$$ and consequently for $|\strata_{\AI}/G|$ since $G$ acts uniformly of 
order 2 on $\strata_{\AI}$.  By Proposition \ref{orbitstopoints}, the number of $\F_q$-points in the smooth locus of $\X_r$ is $|\strata_{\AI}/G|/2$.  The resulting formula, counting $\F_q$-points in $\X_r(\overline{\F}_q)$, is $|\strata_Z/G|+|\strata_D/G|+|\strata_\Dbar/G|+|\strata_{\AI}/G|/2=$ $$(q-1)^{r-1}
   \left((q+1)^{r-1}-1\right)
   q^{r-1}+\frac{1}{2} q
   \left((q-1)^{r-1}+(q+1)^{r-1}
   \right),$$ for $q=p^k$ where $p$ is an odd prime.  Thus, since $\X_r$ admits an appropriate spreading out by Seshadri, we conclude that $\X_r$ is polynomial-count, and so by Katz's theorem the counting polynomial is the $E$-polynomial, as claimed.
Likewise we conclude that the $E$-polynomial for $\X_{\Z^r}(\SL_2(\C))=\X_r^{sing}$ is $|\strata_Z/G|+|\strata_D/G|+|\strata_\Dbar/G|=\frac{1}{2}((q-1)^r+(q+1)^r).$
\end{proof}

\section{Final Remarks}\label{finalsection}
As determined in \cite{FlLa}, the Poincar\'e Polynomial for $\X_r$ is 
$$P_{\fr^r}(t)=-\frac{t\left(t^3+1\right)^r}{1-t^4}+\frac{1}{2} t^3\left(\frac{(t+1)^r}{1-t^2}-\frac{(1-t)^r}{t^2+1}\right)+t+1.$$  Evaluating at $t=-1$ gives the Euler characteristic $\chi(\X_r)=2^{r-2}$ (for $r\geq 2$).  In \cite{FlLa2}, it is shown that $\X_r^{sing}=\X_{\Z^r}(\SL_2(\C))$, and in \cite{FlLa4} it is shown that $\X_{\Z^r}(\SL_2(\C))$ is homotopic to $(S^1)^{\times r}/\Z_2$.  However, the cohomology of the latter is generated by $\Z_2$-invariant cocycles, and so is trivial in odd dimensions, and is $\Z^{\binom{r}{2k}}$ in dimension $2k$.  Thus, the Poincar\'e polynomial for $\X^{sing}=\X_{\Z^r}(\SL_2(\C))$ is $P_{\Z^r}(t)=\sum_{k=0}^{\lfloor n/2\rfloor}\binom{n}{2k}t^{2k}$.  Evaluating at $t=-1$, we see that $\chi(\X^{sing})=2^{r-1}$.  Note that this formula simplifies to $\frac{1}{2}\left((1-t)^r+(t+1)^r\right)$, which implies the equations $q^rE_{\Z^r}(1/q)=P_{\Z^r}(q)$ and $q^rP_{\Z^r}(1/q)=E_{\Z^r}(q)$.  The above paragraph gives an alternative proof of Corollary \ref{eulercor} by the 
inclusion-exclusion principle since $\X_r$ and its strata are complex algebraic sets.  Given that we have both the $E$-polynomial and the Poincar\'e polynomial for $\X_r$, it would interesting to try to compute the full mixed Hodge polynomial which encodes them both.

It would also be interesting to work out the $\SL_3(\C)$ case, using Diophantine geometry as is done in this paper, since in this case the characteristic polynomial is a cubic and so elliptic curves are likely to be in play.

\begin{remark}
After this paper appeared, the $E$-polynomial of the $\SL_3(\C)$-character variety of a free group was determined using fibration techniques in \cite{LaMu}, and can be likewise deduced from results in \cite{MoRe} which use arithmetic and combinatorial techniques.
\end{remark}

Lastly, given Remark \ref{splitremark}, and that $\X_r(\SL_n(\C))$ always admits a spreading out over $\Z[1/n]$, we expect the following conjecture to hold.
\begin{conjecture}
$\X_r(G)$ is polynomial-count for any split reductive algebraic $\C$-group $G$.
\end{conjecture}

\section*{Acknowledgments}
The authors thank Jordan Ellenberg, Nicholas Katz, Ben McReynolds, Juan Souto, and David Speyer for helpful conversations.  We also thank Eugene Xia for hosting the second named author in Taiwan while some of this work was completed.  Lastly, we thank the referee whose comments helped make this paper more readable.


\def\cdprime{$''$} \def\cdprime{$''$} \def\cprime{$'$} \def\cprime{$'$}
  \def\cprime{$'$} \def\cprime{$'$}

\end{document}